\documentclass{amsart}
\usepackage[utf8]{inputenc}
\usepackage{amsthm, amsfonts, amssymb}
\usepackage{mathrsfs}
\usepackage{graphicx}
  \newcommand{\respecte}{\ensuremath{\prec}}
  \newcommand{\cum}{\ensuremath{\kappa}}
  \newcommand\id{\ensuremath{\mathrm{id}}}
  \newcommand{\C}{\ensuremath{\mathbb{C}}}%
  \newcommand{\R}{\ensuremath{\mathbb{R}}}%
  \newcommand{\N}{\ensuremath{\mathbb{N}}}%
  \newcommand{\Z}{\ensuremath{\mathbb{Z}}}%

  \newcommand{\egdef}{\ensuremath{\stackrel{\text{\tiny def}}=}}%

\newtheorem{thm}{Theorem}[section]
\newtheorem{lemma}[thm]{Lemma}
\newtheorem{corollaire}[thm]{Corollary}

\theoremstyle{remark}
\newtheorem*{rem}{Remark}

\theoremstyle{definition}
\newtheorem{dfn}{Definition}[section]

\begin{document}
\author{Mikael de la Salle}
\title{Strong Haagerup inequalities with operator coefficients}

\address{D{\'e}partement de Math{\'e}matiques et Applications
  \\ {\'E}cole Normale Sup{\'e}rieure \\ 45 rue d'Ulm \\ 75005 Paris}

\address{Institut de Math{\'e}matiques de Jussieu \\ rue du Chevalleret \\
  75013 Paris}
\thanks{Partially supported by  ANR-06-BLAN-0015}
\email{mikael.de.la.salle@ens.fr}

\begin{abstract}
  We prove a Strong Haagerup inequality with operator coefficients.  
  If for an integer $d$, $\mathscr H_d$ denotes the subspace of the von
  Neumann algebra of a free group $F_I$ spanned by the words of length $d$
  in the generators (but not their inverses), then we provide in this paper
  an explicit upper bound on the norm on $M_n(\mathscr H_d)$, which
  improves and generalizes previous results by Kemp-Speicher (in the scalar
  case) and Buchholz and Parcet-Pisier (in the non-holomorphic setting).
  Namely the norm of an element of the form $\sum_{i=(i_1,\dots ,i_d)} a_i
  \otimes \lambda(g_{i_1} \dots g_{i_d})$ is less than $4^5 \sqrt e
  (\|M_0\|^2+\dots+\|M_d\|^2 )^{1/2}$, where $M_0,\dots, M_d$ are
  $d+1$ different block-matrices naturally constructed from the family
  $(a_i)_{i \in I^d}$ for each decomposition of $I^d \simeq I^l \times
  I^{d-l}$ with $l=0,\dots, d$. 

  It is also proved that the same inequality holds for the norms in the
  associated non-commutative $L_p$ spaces when $p$ is an even integer, $p
  \geq d$ and when the generators of the free group are more generally
  replaced by $*$-free $\mathscr R$-diagonal operators. In particular it
  applies to the case of free circular operators. We also get inequalities
  for the non-holomorphic case, with a rate of growth of order $d+1$ as for
  the classical Haagerup inequality. The proof is of combinatorial nature
  and is based on the definition and study of a symmetrization process for
  partitions.
\end{abstract}
\maketitle

\section*{Introduction}

Let $F_r$ be the free group on $r$ generators and $|\cdot|$ the length
function associated to this set of generators and their inverses. The
left regular representation of $F_r$ on $\ell^2(F_r)$ is denoted by
$\lambda$, and the $C^*$-algebra generated by $\lambda(F_r)$ is
denoted by $C^*_\lambda(F_r)$. In \cite{MR520930} (Lemma 1.4),
Haagerup proved the following result, now known as the Haagerup
inequality: for any function $f:F_r \to \C$ supported by the words of
length $d$,
\begin{equation}
 \label{eq=Haagerup_inequality}
\left\|\sum_{g \in F_r} f(g) \lambda(g) \right\|_{C^*_\lambda(F_r)} \leq
(d+1) \|f\|_2.
\end{equation}

This inequality has many applications and generalizations. It indeed
gives a useful criterion for constructing bounded operators in
$C^*_\lambda(F_r)$ since it implies in particular that for $f:F_r \to
\C$
\[ \left\|\sum_{g \in F_r} f(g) \lambda(g) \right\|_{C^*_\lambda(F_r)} \leq 2 \sqrt{
\sum_{g \in F_r} (|g|+1)^4 |f(g)|^2},\]
and the so-called Sobolev norm $\sqrt{
\sum_{g \in F_r} (|g|+1)^4 |f(g)|^2}$ is much easier to compute that the operator norm of $\lambda(f)=\sum f(g)\lambda(g)$. The groups for which the same kind of inequality holds for some length function (replacing the term $(d+1)$ in \eqref{eq=Haagerup_inequality} by some power of $(d+1)$) are called groups with property RD \cite{MR943303} and have been extensively studied; they play for example a role in the proof of the Baum-Connes conjecture for discrete cocompact lattices of $SL_3(\R)$ \cite{MR1652538}.

Another direction in which the Haagerup inequality was studied and extended is the theory of operator spaces.
It concerns the same inequality when the function $f$ is allowed to take
operator values. This question was first studied by Haagerup and Pisier in
\cite{MR1240608}, and the most complete inequality was then proved by
Buchholz in \cite{MR1476122}. One of its interests is that it gives an
explanation of the $(d+1)$ term in the classical inequality. Indeed, in the
operator valued case, the term $(d+1) \|f\|_2$ is replaced by a sum of
$d+1$ different norms of $f$ (which are all dominated by $\|f\|_2$ when $f$
is scalar valued). More precisely if $S$ denotes the canonical set of generators of $F_r$ and their inverses,  a function $f:F_r \to M_n(\C)$ supported by the words of length $d$ can be
viewed as a family $(a_{(h_1,\dots,h_d)})_{(h_1,\dots, h_d)\in S^d}$ of
matrices indexed by $S^d$ in the following way: $a_{(h_1,\dots, h_d)}
=f(h_1 h_2 \dots h_d)$ if $|h_1\dots h_d|=d$ and $a_{(h_1,\dots, h_d)}=0$
otherwise.

The family of matrices $a=(a_h)_{h \in S^d}$ can be seen in various natural
ways as a bigger matrix, for any decomposition of $S^d \simeq S^l \times
S^{d-l}$.  If the $a_h$'s are viewed as operators on a Hilbert space $H$
($H=\C^n$), then let us denote by $M_l$ the operator from $H \otimes
\ell^2(S)^{\otimes {d-l}}$ to $H \otimes \ell^2(S)^{\otimes l}$ having the
following block-matrix decomposition:
\[M_l = \left(a_{(s,t)}\right)_{s\in S^l,t\in S^{d-l}}.\]

Then the generalization from \cite{MR1476122} is
\begin{thm}[\cite{MR1476122},Theorem 2.8]
  Let $f:F_r \to M_n(\C)$ supported by the words of length $d$ and
  define $(a_h)_{h \in S^d}$ and $M_l$ for $0 \leq l \leq d$ as above.
  Then
\[ \left\|\sum_{g \in W_d} f(g) \otimes \lambda(g) \right\|_{M_n \otimes
  C^*_\lambda(F_r)} \leq \sum_{l=0}^d \|M_l\|.\]
\end{thm}

The same result has also been extended in \cite{MR2136820} to the
$L_p$-norms up to constants that are not bounded as $d \to \infty$. See
also \cite{MR2279097} and \cite{MR2330976}.

More recently and in the direction of free probability, Kemp and Speicher \cite{MR2353703} discovered the striking fact that, whereas the constant $(d+1)$ is optimal in \eqref{eq=Haagerup_inequality}, when restricted to (scalar) functions
supported by the set $W_d^+$ of words of length $d$ in the generators
$g_1,\dots, g_r$ but \emph{not their inverses} (it is the holomorphic
setting in the vocabulary of \cite{MR2174419} and \cite{MR2353703}), this
constant $(d+1)$ can be replaced by a constant of order $\sqrt d$.

\begin{thm}[\cite{MR2353703},Theorem 1.4]
\label{thm=Kemp_Speicher}
Let $f:F_r \to \C$ be a function supported on $W_d^+$. Then 
\[ \left\|\sum_{g \in W_d^+} f(g) \lambda(g) \right\|_{C^*_\lambda(F_r)}
\leq \sqrt e \sqrt{d+1} \|f\|_2.\]
\end{thm}

A similar result has been obtained when the operators $\lambda(g_1),\dots,
\lambda(g_r)$ are replaced by free $\mathscr R$-diagonals elements: Theorem
1.3 in \cite{MR2353703}. These results are proved using combinatorial
methods: to get bounds on operator norms the authors first get bounds for
the norms in the non-commutative $L_p$-spaces for $p$ even integers, and
make $p$ tend to infinity. For an even integer, the $L_p$-norms are
expressed in terms of moments and these moments are studied using the free
cumulants.

In this paper we generalize and improve these results to the
operator-valued case. As for the generalization of the usual Haagerup
inequality, the operator valued inequality we get gives an explanation
of the term $\sqrt{d+1}$: for operator coefficients this term has to
be replaced by the $\ell^2$ combination of the norms $\|M_l\|$
introduced above. A precise statement is the following. We state the
result for the free group $F_\infty$ on countably many generators
$(g_i)_{i\in \N}$, but it of course applies for a free group with
finitely many generators.

\begin{thm} 
  \label{thm=main_result_for_the_free_group}
  For $d \in \N$, denote by $W_d^+ \subset F_\infty$ the set of words
  of length $d$ in the $g_i$'s (but \emph{not their inverses}). For $k
  =(k_1,\dots, k_d) \in \N^d$ let $g_k = g_{k_1} \dots g_{k_d} \in
  W_d^+$.

  Let $a = (a_k)_{k \in \N^d}$ be a finitely supported family of
  matrices, and for $0\leq l \leq d$ denote by
  $M_l=\left(a_{(k_1,\dots,k_l),(k_{l+1},\dots,k_d)} \right)$ the
  corresponding $\N^l \times \N^{d-l}$ block-matrix. Then
\begin{equation} 
\label{eq=haagerup_holomorphe_operator_free_group}
\left\| \sum_{k \in \N^d} a_k \otimes \lambda(g_k)
  \right\| \leq 4^5 \sqrt e \left( \sum_{l=0}^d \|M_l\|^2\right)^{1/2}.
\end{equation}
\end{thm}

Note that even when $a_k \in \C$, this really is (up to the constant $4^5$) an
improvement of Theorem \ref{thm=Kemp_Speicher}. Indeed it is always true
that for any $l$, $\|M_l\|^2 \leq Tr(M_l^* M_l) = \sum_k |a_k|^2$. There is
equality when $l=0$ or $d$ but the inequality is in general strict when $0
< l < d$. Thus if the $a_k$'s are scalars such that $\|(a_k)\|_2=1$ and
$\|M_l\|\leq 1/\sqrt d$ for $0<l<d$, the inequality in Theorem
\ref{thm=main_result_for_the_free_group} becomes $\left\| \sum_{k \in \N^d}
  a_k \lambda(g_k) \right\| \leq 4^5 \sqrt{3e} \|(a_k)\|_2$. Since the
reverse inequality $\left\| \sum_{k \in \N^d} a_k \lambda(g_k) \right\|
\geq \|(a_k)\|_2$ always holds, we thus get that $\left\| \sum_{k \in \N^d}
  a_k \lambda(g_k) \right\| \simeq \|(a_k)\|_2$ with constants independent
of $d$. An example of such a family is given by the following construction:
if $p$ is a prime number and $a_{k_1,\dots,k_d} = \exp(2i \pi k_1 \dots
k_d/p)/p^{d/2}$ for any $k_i \in \{1,\dots ,p\}$ and $a_k=0$ otherwise then
obviously $\sum_k |a_k|^2=1$, whereas a computation (see Lemma
\ref{lemma=norm_of_matrices_for_p_prime}) shows that $\|M_l\|^2 \leq d/p$
if $0 < l<d$. It is thus enough to take $p \geq d^2$.

As in \cite{MR2353703}, the same arguments apply for the more general
setting of $*$-free $\mathscr R$-diagonal elements ($*$-free means that the
$C^*$-algebras generated are free). Moreover we get significant results
already for the $L_p$-norms for $p$ even integers. Recall that on a
$C^*$-algebra $\mathcal A$ equipped with a trace $\tau$, the $p$-norm of an
element $x \in \mathcal A$ is defined by $\|x\|_p= \tau(|x|^p)^{1/p}$ for
$1\leq p <\infty$, and that for $p=\infty$ the $L^\infty$ norm is just the
operator norm. In the following the algebra $M_n\otimes \mathcal A$ will be
equipped with the trace $Tr \otimes \tau$. The most general statement we
get is thus:

\begin{thm}
\label{thm=Main_theorem} 
Let $c$ be an $\mathscr R$-diagonal operator and $(c_k)_{k \in \N}$ a
family of $*$-free copies of $c$ on a tracial $C^*$-probability space
$(\mathcal A,\tau)$. Let $(a_k)_{k \in \N^d}$ be as above a finitely
supported family of matrices and
$M_l=\left(a_{(k_1,\dots,k_l),(k_{l+1},\dots,k_d)} \right)$ for $0\leq
l \leq d$ the corresponding $\N^l \times \N^{d-l}$ block-matrix.

For $k =(k_1,\dots, k_d) \in \N^d$ denote $c_k = c_{k_1} \dots c_{k_d}$.

Then for any integer $m$,
  \begin{equation}
\label{eq=haagerup_holomorphe_operator_Rdiag}
 \left\| \sum_{k \in \N^d} a_k \otimes c_k \right\|_{2m}
    \leq 4^5 \|c\|_2^{d-2} \|c\|_{2m}^{2} e \sqrt{1 + \frac d m} \left(
      \sum_{l=0}^d \|M_l\|_{2m}^2\right)^{1/2}.
\end{equation}

For the operator norm,
\begin{equation} \left\| \sum_{k \in \N^d} a_k \otimes c_k \right\| \leq
  4^5 \|c\|_2^{d-2} \|c\|^{2} \sqrt{e} \left( \sum_{l=0}^d
    \|M_l\|^2\right)^{1/2}.
\end{equation}

When the $c_k$'s are circular these inequalities are valid without
the factor $4^5 \|c\|_2^{d-2} \|c\|^{2}$.
\end{thm}

The outline of the proof of Theorem \ref{thm=Main_theorem} is the same
as the proof of Theorem 1.3 in \cite{MR2353703}: we first prove the
statement for the $L_p$-norms when $p=2m$ is an even integer (letting
$p \to \infty$ leads to the statement for the operator norm). This is
done with the use of free cumulants that express moments in terms of
non-crossing partitions (the definition of non-crossing partitions is
recalled in part \ref{susbection=study_of_NCstar_d_m}). More precisely
to any integer $n$, any non-crossing partition $\pi$ of the set
$\{1,\dots,n\}$ and any family $b_1,\dots,b_n \in \mathcal A$ the free
cumulant $\cum_\pi[b_1,\dots,b_n] \in \C$ is defined (see
\cite{MR2266879} for a detailed introduction). When $\pi=1_n$ is the
partition into only one block, $\cum_\pi$ is denoted by $\cum_n$. The
free cumulants have the following properties:
\begin{itemize}
\item {\it Multiplicativity:} If $\pi = \{V_1,\dots,V_s\}$,
  $\cum_\pi[b_1,\dots,b_n] = \prod_i \cum_{|V_i|}[(b_k)_{k \in V_i}]$.
\item {\it Moment-cumulant formula:} $\tau(b_1,\dots,b_n) = \sum_{\pi \in
    NC(n)} \cum_\pi[b_1,\dots,b_n]$.
\item {\it Characterization of freeness:} A family $(\mathcal A_i)_i$ of
  subalgebras is free iff all mixed cumulants vanish, \emph{i.e.} for any
  $n$, any $b_k \in \mathcal A_{i_k}$ and any $\pi \in NC(n)$ then
  $\cum_\pi[b_1,\dots,b_n]=0$ unless $i_k=i_l$ for any $k$ and $l$ in a
  same block of $\pi$.
\end{itemize}
The first two properties characterize the free cumulants (and hence could be
taken as a definition), whereas the third one motivates their use in free
probability theory. Since the $*$-distribution of an operator $c\in
(\mathcal A,\tau)$ is characterized by the trace of the polynomials in $c$
and $c^*$, the cumulants involving only $c$ and $c^*$ (that is the
cumulants $\cum_\pi[(b_i)]$ with $b_i \in \{c,c^*\}$ for any $i$) depend
only on the $*$-distribution of $c$.

In order to motivate the combinatorial study of certain non-crossing
partitions in the first section, let us shortly sketch the proof of the
main result. For details, see part \ref{part=holomorphic_setting_proof}.
With the notation of Theorem \ref{thm=Main_theorem} let $A=\sum a_k \otimes
c_k$. For $k=(k(1),\dots,k(d))\in \N^d$ set $\widetilde a_k =
a_{(k(d),\dots,k(1)}$ and $\widetilde c_k = c_{k(d)}\dots c_{k(1)}$ so that
$(\widetilde c_k)^*=c_{k(1)}^*\dots c_{k(d)}^*$. Then $A^* = \sum_k
\widetilde a_k^* \otimes \widetilde c_k^*$, and for $p=2m$ the $p$-th power
of the $p$-norm of $A$ is just the trace $Tr \otimes \tau$ of $(A A^*)^m$,
which can be expressed by linearity as the sum of the terms of the form
$Tr(a_{k_1} \widetilde a_{k_2}^* \dots a_{k_{2m-1}} \widetilde
a_{k_{2m}}^*) \otimes \tau(c_{k_1} \widetilde c^*_{k_2} \dots c_{k_{2m-1}}
\widetilde c_{k_{2m}}^*)$. The expression $c_{k_1} \widetilde c^*_{k_2}
\dots c_{k_{2m-1}} \widetilde c_{k_{2m}}^*$ is the product of $2dm$ terms
of the form $c_i$ or $c_i^*$ (for $i \in \N$). Apply the moment-cumulant
formula to its trace. Using the characterization of freeness with cumulants
and then the multiplicativity of cumulants and the fact that cumulants only
depend on the $*$-distribution we thus get
\[
\left\| \sum_{k \in \N^d} a_k \otimes c_k \right\|_{2m}^{2m} = \sum_{\pi
  \in NC(2dm)} \cum_\pi[c_{d,m}] \underbrace{\sum_{(k_1,\dots, k_{2m})
    \respecte \pi} Tr(a_{k_1} \widetilde a_{k_2}^* \dots \widetilde
  a_{k_{2m}}^*)}_{ \egdef S(a,\pi,d,m)},
\]
where for $k \in \N^{2dm}$ and $\pi \in NC(2dm)$ we write $k \respecte \pi$ if $k_i=k_j$ whenever $i$ and $j$ belong to the same block of $\pi$ and where
\[c_{d,m} = \overbrace{\underbrace{c, \dots, c}_{d} ,\underbrace{c^*, \dots,
    c^*}_{d}, \dots ,\underbrace{c, \dots ,c}_{d}, \underbrace{c^*, \dots,
    c^*}_{d}}^{2m\textrm{ groups}}.
\]

Up to this point we did not use the assumption that $c$ is $\mathcal
R$-diagonal. But as in \cite{MR2353703}, since the $\mathcal R$-diagonal
operators are those operators for which the list of non-zero cumulants is
very short (see part \ref{part=holomorphic_setting_proof} for details), we
get that the previous sum can be restricted to a sum over the partitions in
the subset $NC^*(d,m)\subset NC(2dm)$, which is defined and extensively
studied in part \ref{susbection=study_of_NCstar_d_m}:
\begin{equation}
 \label{eq=formula_of_2m_norms_in_terms_of_cumulants}
  \left\| \sum_{k \in \N^d} a_k \otimes c_k \right\|_{2m}^{2m}  = 
  \sum_{\pi \in NC^*(d,m)} \cum_\pi[c_{d,m}] S(a,\pi,d,m).
\end{equation}
The term $\cum_\pi[c_{d,m}]$ is easy to dominate (Lemma \ref{lemma=domination_of_cumulants_with_many_pairs}). When the $a_k$'s are scalars the second term $S(a,\pi,d,m)$ can be dominated by $\|(a_k)\|_{\ell^2}^{2m}$ (by the usual Cauchy-Schwarz inequality). This is what is done in the proof of \cite{MR2353703}. But here the fact that we are dealing with operators and not scalars forces
to derive a more sophisticated Cauchy-Schwarz type inequality that may
control explicitly the expressions $S(a,\pi,d,m)$ in terms of norms of the
operators $M_l$. This is one of the main technical results in this paper,
Corollary \ref{corollary=operator_Cauchy-Schwarz_inequality}. This Corollary states that 
\begin{equation}
\label{eq=majoration_ultime_de_S}
  |S(a,\pi,d,m)| \leq \prod_{l=0}^d \left\|M_l \right\|_{2m}^{2m\mu_l}
\end{equation}
for some non-negative $\mu_l$ with $\sum_l \mu_l=1$. Moreover the $\mu_l$
are explicitely described by some combinatorial properties of $\pi$.  This
inequality is proved through a process of ``symmetrization'' of partitions.
The basic observation is that if one applies a simple Cauchy-Schwarz
inequality to $S(a,\pi,d,m)$ (Lemma \ref{lemma=CauchySchwarz}), this
corresponds on the level of partitions to a certain combinatorial operation
of symmetrization that is studied in the part
\ref{part=def_and_observations}. This observation was already used
implicitely in \cite{MR1812816}, Lemma 2, in some special case: Buchholz
indeed notices that for $d=1$ and if $\pi$ is a pairing (\emph{i.e.} has
blocks of size $2$), this Cauchy-Schwarz inequality corresponds to some
transformation of pairings (for which he does not give a combinatorial
description), and that iterating this inequality eventually leads to an
domination of the form \eqref{eq=majoration_ultime_de_S} (for $d=1$) but in
which he does not compute the exponents $\mu_0$ and $\mu_1$.

In our more general setting it also appears that repeating this operation
in an appropriate way turns every non-crossing partition $\pi \in
NC^*(d,m)$ into one very simple and fully symmetric partition for which the
expression $S(a,\pi,d,m)$ is exactly the ($2m$-power of the $2m$-) norm of
one of the $M_l$'s.  This is stated and proved in Corollary
\ref{corollary=symmetrization_for_n} and Lemma \ref{lemma=identifications}.
One important feature of our study of the symmetrization operation on $NC^*(d,m)$ is the fact that we are able to determine some combinatorial invariants of this operation (see part
\ref{part=invariants_of_P_k}). This allows to keep track of the exponents of
the $\|M_l\|_{2m}$ that progressively appear during the symmetrization
process, and to compute the coefficients $\mu_l$ in \eqref{eq=majoration_ultime_de_S}.

The second technical result that we prove and use is a finer study of
$NC^*(d,m)$. The main conclusion is Theorem
\ref{thm=decomposition_of_NC*nm} which expresses that partitions in
$NC^*(d,m)$ have mainly blocks of size $2$ and that $NC^*(d,m)$ is not very
far from the set $NC(m)^{(d)}$ of non-decreasing chains of non-crossing
partitions on $m$ (in the sense that there is a natural surjection
$NC^*(d,m) \to NC(m)^{(d)}$ such that the fiber of any point has a
cardinality dominated by a term not depending on $d$). This combinatorial result is then generalized in Theorem \ref{thm=decomposition_of_NCn_m} and Lemma \ref{thm=description_of_NCpart_withKintervals},
and then used to transform the sum in \eqref{eq=formula_of_2m_norms_in_terms_of_cumulants} into a sum over $NC(m)^{(d)}$ for which the combinatorics are well known by \cite{MR583216}.

We prove also the following results, which are extensions to the
non-holomorphic case of the previous results and their proofs. Let $c$ be
an $\mathscr R$-diagonal operator and $(c_k)_{k \in \N}$ a family of
$*$-free copies of $c$ on a tracial $C^*$-probability space $(\mathcal
A,\tau)$. For $\varepsilon =
(\varepsilon_1,\dots,\varepsilon_d)\in\{1,*\}^d$ and
$k=(k_1,\dots,k_d)\in\N^d$ denote $c_{k,\varepsilon} =
c_{k_1}^{\varepsilon_1} \dots c_{k_d}^{\varepsilon_d}$. The result is an
extension of Haagerup's inequality for the space generated by the
$c_{k,\varepsilon}$ for the $k,\varepsilon$ satisfying $k_i=k_{i+1}
\Rightarrow \varepsilon_i = \varepsilon_{i+1}$, \emph{i.e.} for which
$\lambda(g)_{k,\varepsilon}$ has length $d$. Denote by $I_d$ the set of
such $(k,\varepsilon)$.

\begin{thm}\label{thm=non_holomo_Rdiag}
  Let $(a_{(k,\varepsilon)})_{(k,\varepsilon) \in (\N \times\{1,*\})^d}$ be
  a finitely spported family such that $a_{(k,\varepsilon)}=0$ for
  $(k,\varepsilon) \notin I_d$. For $0 \leq l \leq d$, let $M_l$ be the
  matrix formed as above from $(a_{(k,\varepsilon)})$ for the decomposition
  $(\N \times\{1,*\})^d = (\N \times\{1,*\})^l \times (\N
  \times\{1,*\})^{d-l}$.

  Then for any $p \in 2\N \cup \{\infty\}$
  \[ \left\|\sum_{(k,\varepsilon) \in (\N \times\{1,\varepsilon\})^d}
    a_{k,\varepsilon} \otimes c_{k,\varepsilon}\right\|_p \leq 4^5
  \|c\|_p^2 \|c\|_2^{d-2} (d+1) \max_{0 \leq l\leq d} \|M_l\|_p.\]
\end{thm}

Similarly for self-adjoint operators we have:
\begin{thm}\label{thm=non_holomo_selfadj}
  Let $\mu$ be a symmetric compactly supported probability measure on $\R$,
  and $c$ a self-adjoint element of a tracial $C^*$-algebra distributed as
  $\mu$.

  Let $(c_k)_{k\in\N}$ be self-adjoint free copies of $c$ and
  $(a_{k_1,\dots,k_d})_{k_1,\dots ,k_d \in \N}$ be a finitely supported
  family of matrices such that $a_{k_1,\dots,k_d}=0$ if $k_i = k_{i+1}$ for
  some $1\leq i<d$. Then for any $p \in 2\N\cup \{\infty\}$
  \begin{equation}
\label{eq=non_holomo_selfadj}
 \left\|\sum_{(k_1,\dots,k_d) \in \N^d} a_{k_1,\dots,k_d} \otimes
    c_{k_1} \dots c_{k_d} \right\|_p \leq 4^5 \|c\|_p^2 \|c\|_2^{d-2} (d+1)
  \max_{0 \leq l\leq d} \|M_l\|_p.
\end{equation}
\end{thm}

For the case of the semicircular law and scalar coefficient $a_k$, this
result is not new. It is due to Bo{\.z}ejko \cite{MR1146011}, and was reproved
using combinatorial methods by Biane and Speicher, Theorem 5.3.4 of
\cite{MR1660906}. Our proof is a generalization of their proof and uses
it. Note also that the condition that $a_{k_1,\dots,k_d}=0$ if
$k_i=k_{i+1}$ for some $i$ is crucial to get \eqref{eq=non_holomo_selfadj}:
indeed if $a_{k_1,\dots,k_d}=0$ except for $a_{1,\dots,1}=1$ then we have
the equality $\|\sum_{k \in \N^d} a_k \otimes c_{k_1} \dots c_{k_d}\|_p =
\|c_1^d\|_p = \|c\|_{dp}^d$, whereas $\max_l \|M_l\|_p = 1$ and if $\mu$ is
not a Bernoulli measure $\|c\|_p^2 \|c\|_2^{d-2} (d+1) =o(\|c\|_{dp}^d)$
when $d \to \infty$. The inequality \eqref{eq=non_holomo_selfadj} thus does
not hold for this choice of $(a_k)$, even up to a constant.

These results are of some interest since they prove a new version of
Haagerup's inequality in a broader setting, but they are still
unsatisfactory since one would expect to be able to replace the term $(d+1)
\max_{0 \leq l \leq d} \|M_l\|$ by $\sum_{l=0}^d \|M_l\|$.

The paper is organized as follows: the first part only deals with
combinatorics of non-crossing partitions. In the second part we use the
results of the first part to get inequalities for the expressions
$S(a,\pi,d,m)$. In the third and last part we finally prove the main
results stated above.

Although some definitions are recalled, the reader will be assumed to be familiar with the basics of free probability theory and more precisely to its combinatorial aspect (non-crossing partitions, free cumulants, $\mathcal R$-diagonal operators...). For more on this see \cite{MR2266879}. For the vocabulary of non-commutative $L^p$ spaces nothing more than the definitions of the $p$-norm, the Cauchy-Schwarz inequality $|\tau(ab)| \leq \|a\|_2 \|b\|_2$ and the fact that $\|x\|_\infty = \lim_{p \to \infty} \|x\|_p$ will be used.

\section{Symmetrization of non-crossing partitions}
\label{section=symmetrization}
For any integer $n$, we denote by $[n]$ the interval $\{1,2,\dots, n\}$,
which we identify with $\Z/n\Z$ and which is endowed with the natural
cyclic order: for $k_1,\dots, k_p \in [n]$ we say that $k_1 < k_2 < \dots
<k_p$ for the cyclic order if there are integers $l_1,\dots l_p$ such that
$l_1< l_2 < \dots <l_p$, $k_i = l_i \mod n$ and $l_p-l_1 \leq n$. In other
words, if the elements of $[n]$ are represented on the vertices of a
regular polygon with $n$ vertices labelled by elements of $[n]$ as in
Figure \ref{picture=graphrep_of_partition}, then we say that $k_1 < k_2 <
\dots <k_p$ if the sequence $k_1,\dots k_p$ can be read on the vertices of
the regular polygon by following the circle clockwise for at most one full
circle.

If $\pi$ is a partition of $[n]$, and $i\in [n]$, the element of the
partition $\pi$ to which $i$ belongs is denoted by $\pi(i)$. We also write
$i \sim_\pi j$ if $i$ and $j$ belong to the same block of the partition
$\pi$.

If the elements of $[n]$ are represented on the vertices of a regular
polygon with $n$ vertices, a partition $\pi$ of $[n]$ is then represented
on the regular polygon by drawing a path between $i$ and $j$ if $i \sim_\pi
j$. See Figure \ref{picture=graphrep_of_partition} for an example.
\begin{figure}[!ht]
  \center
  \includegraphics{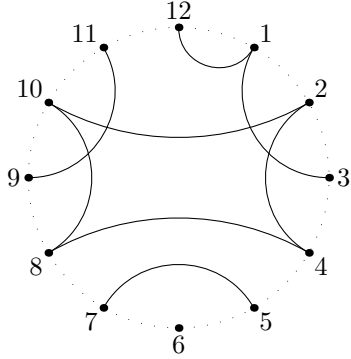}
  \caption{A graphical representation of the partition
    $\left\{\{1,3,12\},\{2,4,8,10\},\{5,7\},\{6\},\{9,11\}\right\}$.}\label{picture=graphrep_of_partition}
\end{figure}
\subsection{Definitions and first observation}
\label{part=def_and_observations}
We introduce the operations $P_k$ on the set of partitions of an even
number $n=2N$. This definition is motivated by Lemma
\ref{lemma=CauchySchwarz}.

\begin{dfn}\label{def=intervals_symmetries}
  Let $k \in [2N]$, and $I_k$ the subinterval of $[2N]$ of length $N$ and
  ending with $k$, $I_k=\{k-N+1,k-N+2, \dots ,k\}$ and $s_k^{(N)}$ (or
  simply $s_k$ when no confusion is possible) the symmetry $s_k(i)=2k+1-i$
  (note that $s_k$ is an involution of $[2N]$ that exchanges $I_k$ and
  $[2N] \setminus I_k$). For a partition $\pi$ of $[2N]$, $s_k(\pi)$ is the
  symmetric of $\pi$: $A \in s_k(\pi)$ if $s_k^{-1}(A) = s_k(A) \in
  \pi$. In other words $i \sim_{s_k(\pi)} j$ if and only if $s_k(i)
  \sim_\pi s_k(j)$.

\begin{figure}
  \center
  \includegraphics{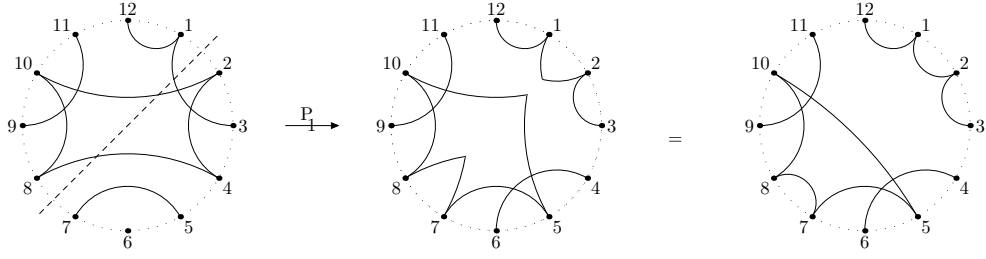}
  \caption{The operation $P_1$ on the partition
    $\left\{\{1,3,12\},\{2,4,8,10\},\{5,7\},\{6\},\{9,11\}\right\}$.}\label{picture=graphrep_of_P_1}
\end{figure}

For any partition $\pi$ of $[2N]$, we denote by $P_k(\pi)$ the partition of
$[2N]$ that we view as a symmetrization of $\pi$ around $k$, and which is
formally defined by the following: if one denotes $\pi'=P_k(\pi)$, then
\begin{eqnarray}
  \textrm{for } i,j\in I_k & i \sim_{\pi'} j \textrm{ if and only if } i
  \sim_\pi j\\
  \textrm{for } i,j\in [2N]\setminus I_k & i \sim_{\pi'} j \textrm{ if and
    only if } s_k(i)
  \sim_\pi s_k(j)\\
  \textrm{for } i \in I_k, j\notin I_k & i \sim_{\pi'} j \textrm{ if and only
    if } i
  \sim_\pi s_k(j) \textrm{and } \exists l\notin I_k, i \sim_\pi l.
\end{eqnarray}
\end{dfn}
It is straightforward to check that this indeed defines a partition of
$[2N]$, and that it is symmetric with respect to $k$, that is
$s_k(\pi')=\pi'$.

The operation $P_k$ is perhaps more easily described graphically: represent
$\pi$ on a regular polygon as above, and draw the symmetry line going
through the middle of the segment $[k,k+1]$. A graphical representation of
$P_k(\pi)$ is then obtained by erasing all the half-polygon not containing
$k$ and replacing it by the mirror-image of the half-polygon containing
$k$. See Figure \ref{picture=graphrep_of_P_1} for an example.

The following lemma expresses the fact that applying sufficiently many
times appropriate operators $P_k$, one can make a partition symmetric
with respect to all the $s_k$'s. See Figure
\ref{picture=symmetrization_process} to see an example of this
symmetrization process.
\begin{lemma}\label{lemma=symmetrization}
  Let $m$ be a positive integer.

  Let $k \in \N$ such that $2^{k} \geq m$. Then for any partition $\pi$ of
  $[2m]$, the partition $\pi_k = P_{2^k} P_{2^{k-1}} \dots P_{2} P_{1}
  P_{m} (\pi)$ is one of the four following partitions:
\begin{eqnarray*}
  \pi_k = 0_{2m} & = &\left\{\{j\},j\in [2m]\right\}\\
  \pi_k = c_m &= &\left\{\{2j;2j+1\},j\in [m]\right\}\\
  \pi_k = r_m &= &\left\{\{2j-1;2j\},j\in [m]\right\}\\
  \pi_k = 1_{2m} &= &\{[2m]\}
\end{eqnarray*} 

\end{lemma}
\begin{figure}
  \center
  \includegraphics{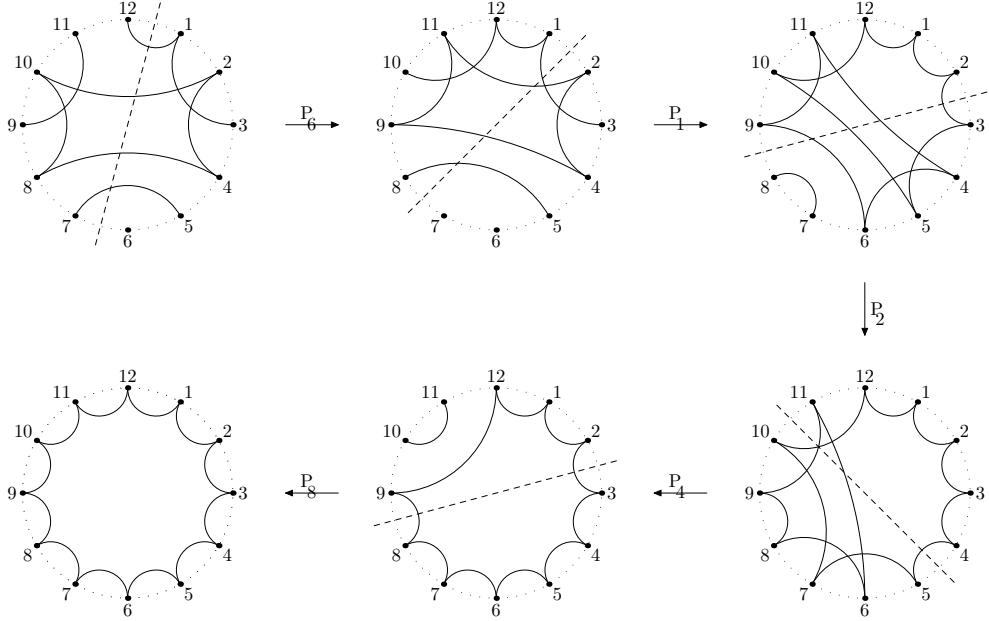}
  \caption{The symmetrization process starting from the partition
    $\left\{\{1,3,12\},\{2,4,8,10\},\{5,7\},\{6\},\{9,11\}\right\}$.}\label{picture=symmetrization_process}
\end{figure}
\begin{proof}
  Let $A= I_m \cap \pi(1) \setminus \{1\}$ and $B=([2m]\setminus I_m) \cap
  \pi(1) $. The four cases correspond respectively to the four following
  cases:
\begin{enumerate}
\item $A=B=\emptyset$.
\item $A=\emptyset$ and $B \neq \emptyset$.
\item $A \neq \emptyset$ and $B = \emptyset$.
\item $A \neq \emptyset$ and $B \neq \emptyset$.
\end{enumerate}

\begin{figure}
  \center
  \includegraphics{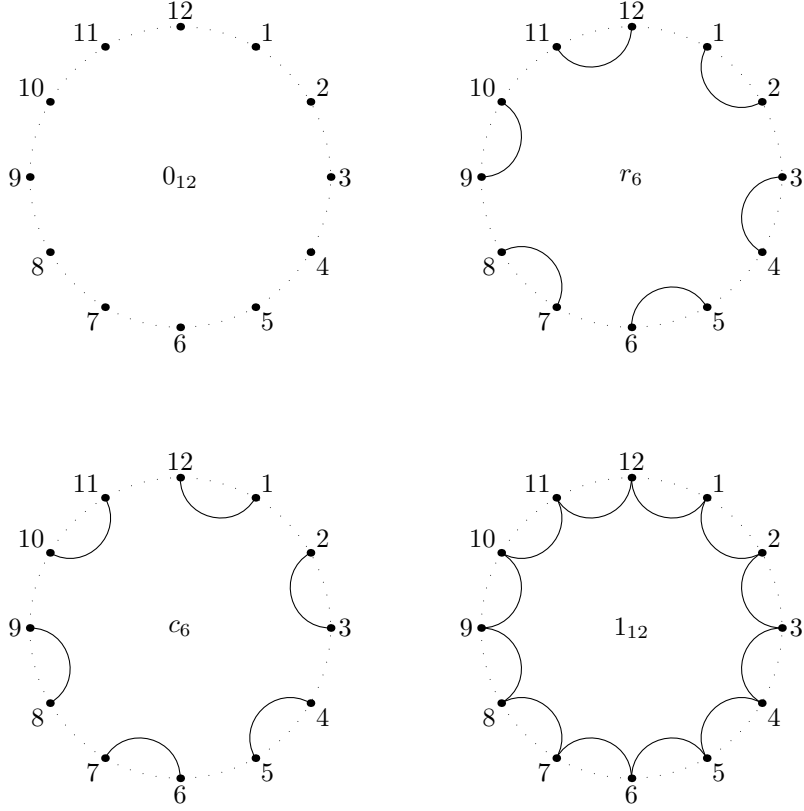}
  \caption{The partitions $0_{12}$, $r_6$, $c_6$ and $1_{12}$.}\label{picture=the_symmetric_partitions}
\end{figure}

In the first case, it is straightforward to prove by induction on $k$ that
$\pi_k$ includes the blocks $\{i\}$ for any $i\in \{1,\dots, 2^{k+1}\}$.

If $A=\emptyset$ and $B \neq \emptyset$, then $P_m(\pi)$ includes the
block $\{0,1\}$ and this implies that $P_1P_m(\pi)$ includes the
blocks $\{0,1\}$ and $\{2,3\}$, which in turn implies that
$P_2P_1P_m(\pi)$ includes the blocks $\{0,1\}$,$\{2,3\}$ and
$\{4,5\}$... More generally $\pi_k$ includes the blocks
$\{0,1\},\{2,3\},\dots,\{2^{k+1},2^{k+1}+1\}$ (this can be proved by
induction). For $2^{k+1} \geq 2m$ this is exactly $\pi_k=c_m$. We
leave the details to the reader.

In the same way, in the third case it is easy to prove by induction on $k$
that $\pi_k$ includes the blocks $\{2l-1,2l\}$ for $l\in\{1,\dots,
2^{k}\}$.

The fourth case follows from a similar proof by induction that $\pi_k(1)$
contains $\{0,1,2,\dots, 2^{k+1}+1\}$. The details are not provided.
\end{proof}

Although $P_k(\pi)$ is defined for any partition $\pi$, we will be mainly
interested in the case when $\pi$ is a non-crossing partition, and more
precisely when $\pi \in NC^*(d,m)$.

\subsection{Study of $NC^*(d,m)$}
\label{susbection=study_of_NCstar_d_m}

We first recall the definition of a non-crossing partition. A partition
$\pi$ of $[N]$ is called non-crossing if for any distinct $i<j<k<l \in
[N]$, $i \sim_\pi k$ and $j \sim_\pi l$ implies $i \sim_\pi j$ (in this
definition either take for $<$ the usual order on $\{1,\dots ,N\}$ or the
cyclic order since it gives to the same notion). More intuitively $\pi$ is
non-crossing if and only if there is a graphical representation of $\pi$ (on a
regular polygon with $n$ vertices as explained in the beginning of section
\ref{section=symmetrization}) such that the paths lie inside the polygon
and only intersect (possibly) at the vertices of the regular polygon. For
example the partitions of Figures \ref{picture=graphrep_of_partition},
\ref{picture=graphrep_of_P_1} are crossing, whereas the partitions in
Figures \ref{picture=the_symmetric_partitions},
\ref{picture=example_of_NCdm}, \ref{picture=exemple_of_map_P} are all
non-crossing. The set of non-crossing partitions of $[N]$ is denoted by
$NC(N)$. The cardinality of $NC(N)$ is known to be equal to the Catalan
number $(2N)!/(N! (N+1)!)$ (see \cite{MR0309747}), but we will only use
that it is less that $4^{N-1}$.

Following \cite{MR2353703}, we introduce the subset $NC^*(d,m)$ of
$NC(2dm)$.

In the following, for a real number $x$ one denotes by $\lfloor{x}\rfloor$
the biggest integer smaller than or equal to $x$. 

Divide the set $[2dm]$ into $2m$ intervals $J_1 \dots J_{2m}$ of size $d$:
the first one is $J_1 = \{1,2\dots ,d\}$, and the $k$-th is $J_k =
\{(k-1)d+1,\dots ,k d\}$.

To each element of $[2dm]$ we assign a label in $\{1,\dots ,d\}$ in the
following way: in any interval $J_k$ of size $d$ as above, the elements are
labelled from $1$ to $d$ if $k$ is odd and from $d$ to $1$ if $k$ is
even. We shall denote by $A_k$ the set of elements labelled by $k$.

\begin{dfn}A non-crossing partition $\pi$ of $[2dm]$ belongs to $NC^*(d,m)$
  if each block of the partition has an even cardinality, and if within
  each block, two consecutive elements $i$ and $j$ belong to intervals of
  size $d$ of different parity. Formally, the last condition means that
  $\lfloor{(i-1)/d}\rfloor \neq \lfloor{(j-1)/d}\rfloor \mod 2$ or
  equivalently $k(i) \neq k(j) \mod 2$ when $i \in J_{k(i)}$ and $j \in
  J_{k(j)}$.
\end{dfn}

Here are some first elementary properties of $NC^*(d,m)$:

\begin{lemma}
\label{lemma=equivalent_def_of_NCstardm}
  If $d=1$, a non-crossing partition $\pi \in NC(2m)$ belongs to
  $NC^*(1,m)$ if and only if it has blocks of even cardinality.

  A non-crossing partition of $[2dm]$ is in $NC^*(d,m)$ if and only
if it has blocks of even cardinality and it connects only elements with the
same labels (\emph{i.e.} it is finer than the partition $\{A_1,\dots ,
A_d\}$).
\end{lemma}
\begin{proof}
  The first statement is a particular case of the second statement, which
  we now prove. For any $i \in [2dm]$ denote by $k(i)$ the integer such
  that $i \in J_{k(i)}$: $k(i)=1+\lfloor{(i-1)/d}\rfloor$. Let $\pi \in
  NC^*(d,m)$. Then by the definition of $NC^*(d,m)$ every block of $\pi$
  contains as many elements $i$ such that $k(i)$ is odd than elements $i$
  such that $k(i)$ is even. We have to prove that if $s$ and $t$ are two
  consecutive elements of a block of $\pi$, then $s$ and $t$ have the same
  labellings. Assume for example that $s$ belongs to an odd interval,
  \emph{i.e.}  $k(s)$ is odd, and denote by $l(s)$ the label of $s$. Then
  $s = (k(s)-1)d + l(s)$. In the same way, $k(t)$ is then even and if
  $l(t)$ is the label of $t$, we have that $t = k(t)d + 1 - l(t)$. Hence
  the number of elements $i \in \{s+1, \dots , t-1\}$ such that $k(i) \big(
  = 1+ \lfloor{(i-1)/d}\rfloor\big)$ is odd is equal to $d-l(s) + d \cdot
  (k(t) -k(s)-1)/2$, and the number of elements $i$ such that $k(i)$ is
  even is equal to $d-l(t) + d \cdot (k(t) -k(s)-1)/2$. But since $\pi$ is
  non-crossing, the interval $\{s+1, \dots , t-1\}$ is a union of blocks of
  $\pi$ and therefore contains as many elements $i$ such that $k(i)$ is odd
  than elements $i$ such that $k(i)$ is even. This implies $l(s)=l(t)$. The
  proof is the same if $k(s)$ is even.

  Now assume that $\pi \in NC(dm)$ has blocks of even cardinality and that
  $\pi$ is finer than the partition $\{A_1,\dots , A_d\}$. Let $s$ and $t$
  be two consecutive elements of a block of $\pi$. Then there is $i$ such
  that $s,t \in A_i$. Since $\pi$ is non-crossing and $\pi$ is finer than
  $\{A_1,\dots , A_d\}$, the set $\{s+1, \dots , t-1\} \cap A_i$ is a union
  of blocks of $\pi$, and in particular it has an even cardinality. But
  $\{s+1, \dots , t-1\} \cap A_i$ is the set of elements labelled by $i$ in
  the union of the intervals $J_k$ for $k(s)<k<k(t)$ (for the cyclic
  order). Hence its cardinality is $k(t) - k(s) -1$. Hence $k(t)-k(s)$ is
  odd. Since $s$ and $t$ are arbitrary, this proves that $\pi \in NC^*(d,m)$.  
\end{proof}

\begin{figure}
  \center
  \includegraphics{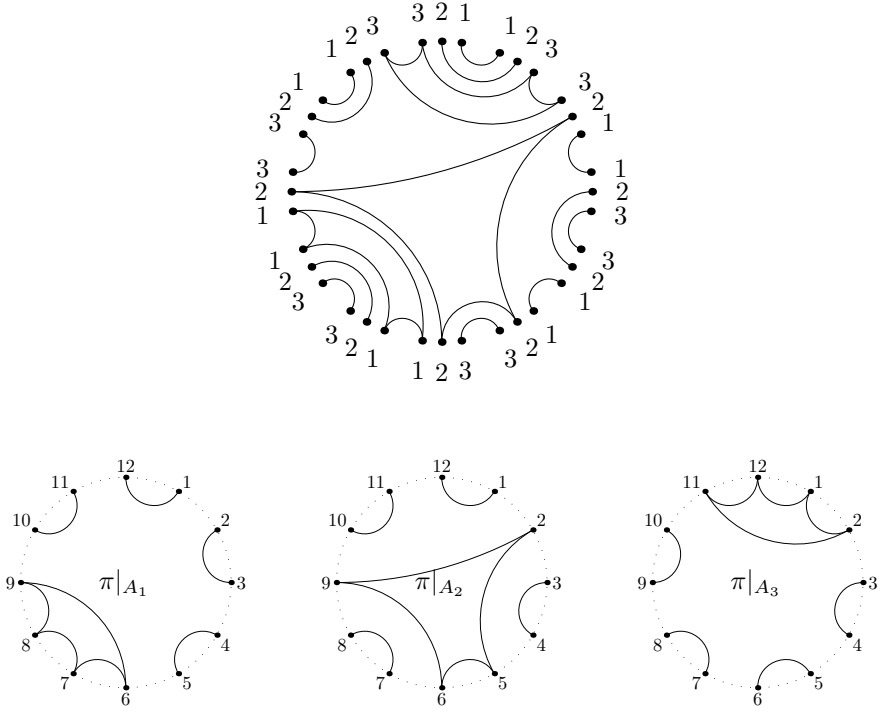}
  \caption{A graphical representation of a partition $\pi$ in $NC^*(3,6)$
    and the corresponding restrictions $\pi{\left|_{A_1}\right.}$,
    $\pi{\left|_{A_2}\right.}$ and  $\pi{\left|_{A_3}\right.}$.}\label{picture=example_of_NCdm}
\end{figure}

Thus to any $\pi \in NC^*(d,m)$ we can assign $d$ partitions
$\pi{\left|_{A_1}\right.},\dots ,\pi{\left|_{A_d}\right.}$, which are the
restrictions of $\pi$ to $A_1,\dots ,A_d$ respectively.  It is immediate
that for any $i\in\{1,\dots ,d\}$, $\pi{\left|_{A_i}\right.} \in NC^*(1,m)$.
See Figure \ref{picture=example_of_NCdm} for an example.
To study $NC^*(d,m)$, we thus begin with the study of $NC^*(1,m)$.

The first lemma shows that if $k$ is a multiple of $d$, then $P_k$ maps
$NC^*(d,m)$ into itself:
\begin{lemma}
\label{lemma=properties_of_symmetrization_forNCnm}
If $k \in \N$ and $\pi \in NC(2N)$ then $P_k(\pi) \in NC(2N)$.

If $k \in \N$ then for any $\pi \in NC^*(d,m)$, $P_{kd}(\pi) \in NC^*(d,m)$.

Moreover if $\pi \in NC^*(d,m)$, then for any $i\in\{1,\dots d\}$:
\[ P_{kd}(\pi){\left|_{A_i}\right.} = P_{k}(\pi{\left|_{A_i}\right.}).\]
\end{lemma}
\begin{proof}[Sketch of Proof]
The first statement is obvious from the graphical point of view: if there
are no crossing, the symmetrization map will not produce one.

The second statement follows from the characterization of Lemma
\ref{lemma=equivalent_def_of_NCstardm}: it is not hard to check that if
$\pi$ has blocks of even cardinality then $P_{kd}(\pi)$ also has. The fact
that $P_{kd}(\pi)$ is finer that $\{A_1,\dots ,A_d\}$ if $\pi$ is follows
from the fact that $s_{kd}(A_i)=A_i$ for any $k$ and $1 \leq i \leq d$.

The third statement follows from the fact that $s_{kd}^{(dm)}$ is
characterized by the properties that for any $1 \leq j \leq 2m$,
$s_{kd}^{(dm)}(J_i) = J_{s_k^{(m)}(i)}$ and $s_{kd}^{(dm)}(A_i)=A_i$ for
$1 \leq i \leq d$.
\end{proof}

We have the following corollary of Lemma \ref{lemma=symmetrization}.
\begin{corollaire}
\label{corollary=symmetrization_for_n}
  Let $\pi \in NC^*(d,m)$. Then for $2^k \geq m$, the partition $\pi_k =
  P_{2^kd} P_{2^{k-1}d} \dots P_{2d} P_{d} P_{md} (\pi)$ is one of the
  $2d+1$ partitions $\sigma_l^{(d,m)}$ for $l=0,1,\dots , d$ and $\widetilde
  \sigma_l^{(d,m)}$ for $l=1,2,\dots ,d$ defined by:
  
\[\sigma_l^{(d,m)}{\left|_{A_i}\right.} = \left\{ 
\begin{array}{ll} c_m & \textrm{if } 1 \leq i\leq l\\
r_m & \textrm{if } l < i\leq d,
\end{array}\right.\]
\[\widetilde \sigma_l^{(d,m)}{\left|_{A_i}\right.} = \left\{ 
\begin{array}{ll} c_m & \textrm{if } 1 \leq i< l\\
1_{2m} & \textrm{if } i=l\\
r_m & \textrm{if } l < i\leq d.
\end{array}\right.\]

Moreover for any integer $i$, $P_{id}(\pi)=\pi$ when $\pi$ is one of the
partitions $\sigma_l^{(d,m)}$ for $l=0,1,\dots , d$ and $\widetilde
\sigma_l^{(d,m)}$ for $l=1,2,\dots ,d$.
\end{corollaire}
\begin{proof}
Let $k$ and $\pi$ as above. By Lemma
\ref{lemma=properties_of_symmetrization_forNCnm},
$\pi_k{\left|_{A_i}\right.} = P_{2^k} P_{2^{k-1}} \dots P_{2} P_{1}
P_{m}(\pi{\left|_{A_i}\right.})$, which is by Lemma
\ref{lemma=symmetrization} one of $0_{2m}$, $r_m$, $c_m$ and $1_{2m}$. But
since $0_{2m}$ does not have blocks of even sizes, only the three $r_m$,
$c_m$ and $1_{2m}$ are possible.

Let $1\leq i<j\leq d$. If $\pi_k{\left|_{A_i}\right.}=r_m$ or $1_{2m}$ then
in particular $i \sim_{\pi_k} 2d+1-i$. Since $\pi_k$ is non-crossing, $j
\nsim_{\pi_k} 1-j$, which implies that $\pi_k{\left|_{A_j}\right.}\neq
c_m,1_{2m}$. Thus $\pi_k{\left|_{A_j}\right.} = r_m$. In the same way if
$\pi_k{\left|_{A_j}\right.}=c_m$ or $1_{2m}$ then
$\pi_k{\left|_{A_i}\right.}=c_m$. This concludes the proof.

Similarly, the second claims follows from the fact (easy to verify) that
$P_i(\pi)=\pi$ for any $i \in [2m]$ when $\pi=1_{2m}$, $r_m$ or $c_m$.
\end{proof}

An important subset of $NC^*(d,m)$ is the subset $NC^*_2(d,m)$ of
partitions in $NC^*(d,m)$ with blocks of size $2$. As explained in part 3.1
of \cite{MR2353703}, $NC^*_2(d,m)$ is naturally in bijection with the
non-decreasing chains (for the natural lattice structure on $NC(m)$) of
length $d$ of non-crossing partitions of $[m]$. Let us denote by
$NC(m)^{(d)}$ this set of non-decreasing chains in $NC(m)$, for the order
of refinement, given by $\pi \leq \pi'$ if $\pi'$ is finer that $\pi$. The
bijective map $NC^*_2(d,m) \to NC(m)^{(d)}$ extends naturally to a (of
course non-bijective) map $NC^*(d,m) \to NC(m)^{(d)}$ which is of
interest. We now describe the construction of this map.

Let $\pi \in NC^*(1,m)$, that is a non-crossing partition of $[2m]$ with
blocks of even size.  Then $\Phi(\pi)$ is the partition of $[m]$ defined by
the fact that $\sim_{\Phi(\pi)}$ is the transitive closure of the relation
that relates $k$ and $l$ if $2k \sim_\pi 2l$ or $2k-1 \sim_\pi 2l$ or $2k
\sim_\pi 2l-1$ or $2k-1 \sim_\pi 2l-1$. That is $\Phi(\pi)$ is the
partition obtained by identifying the $2k-1$ and $2k$ in $[2m]$ to get $k
\in [m]$.

If $\pi\in NC^{*}(d,m)$, we define the map $\mathcal P$ by $\mathcal
P(\pi)=(\Phi(\pi{\left|_{A_1}\right.}),\dots
,\Phi(\pi{\left|_{A_d}\right.}))$. See Figure \ref{picture=exemple_of_map_P}.
\begin{figure}[!ht]
  \center
  \includegraphics{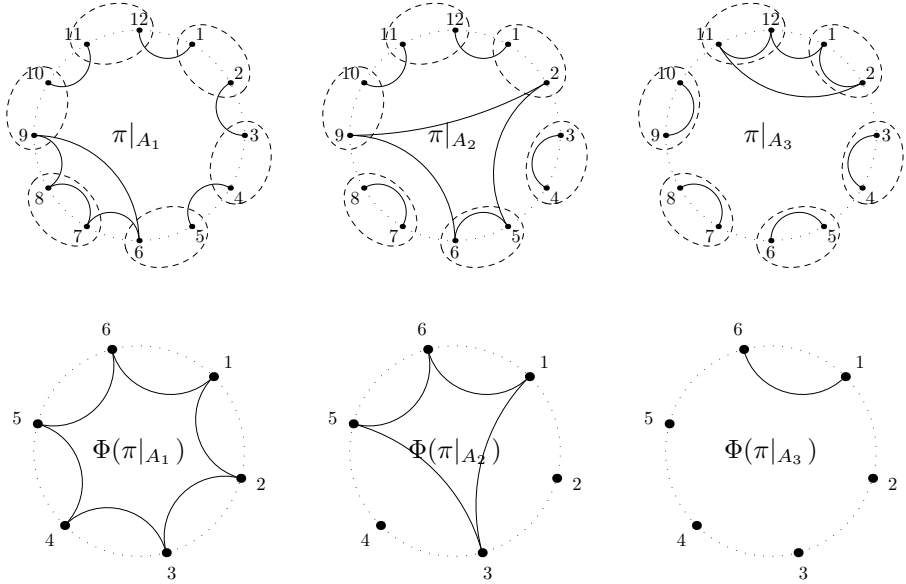}
  \caption{The map $\mathcal P$ for the partition $\pi \in NC^*(3,6)$ of
    Figure \ref{picture=example_of_NCdm}.}\label{picture=exemple_of_map_P}
\end{figure}

The map $\mathcal P$ is a good tool to make a finer study of
$NC^*(d,m)$. 

The main result in this section is that partitions in $NC^*(d,m)$ are not
far from belonging to $NC^*_2(d,m)$:
\begin{thm}
\label{thm=decomposition_of_NC*nm}
  For any $\sigma \in NC^*_2(d,m)$ there are less than $4^{2m}$ partitions
  $\pi\in NC^*(d,m)$ such that $\mathcal P(\pi) = \mathcal P(\sigma)$.

Moreover for such a $\pi$, the partition $\sigma$ is finer than $\pi$ and
the number of blocks of $\pi$ of size $2$ is greater than $dm-2m$, and
every block has size at most $2m$.
\end{thm}
\begin{rem} The remarkable feature of $NC^*(d,m)$ illustrated in this
  Theorem is that the bounds we get on the number of $\pi\in NC^*(d,m)$
  such that $\mathcal P(\pi) = \mathcal P(\sigma)$ and on the elements of
  $[2dm]$ that do not belong to a block of size $2$ of $\pi \in NC^*(d,m)$
  \emph{do not depend on d}.

  In particular since the cardinality of $NC^*_2(d,m)$ is equal to the
  Fuss-Catalan number $1/m \binom{m(d+1)}{m-1}$ which is less that
  $e^m(d+1)^m$ (Corollary 3.2 in \cite{MR2353703}) the first statement of
  the Theorem implies that the cardinality of $NC^*(d,m)$ is less that
  $\big(16 e (d+1)\big)^m$.
\end{rem}

This Theorem will follow from a series of lemmas. Here is the first one,
which treats the case $d=1$:
\begin{lemma}
\label{lemma=taille_du_quotient_pour_n=1}
Let $\sigma \in NC^*_2(1,m)$ and $\pi \in NC^*(1,m)$ such that $\Phi(\pi) =
\Phi(\sigma)$. Then $\sigma$ is finer than $\pi$.

More precisely if $\pi \in NC^*(1,m)$ and if $\{k_1 < k_2 \dots < k_p\}$ is
a block of $\Phi(\pi)$, then for any $i$, $2k_{i} \sim_\pi 2k_{i+1} -1$
(with the convention $k_{p+1} = k_1$).
\end{lemma}
\begin{proof}
  The first statement follows easily from the second one. We thus focus on
  the second statement. At least as far as partitions in $NC^*_2(1,m)$ are
  concerned, this is explained in the discussion preceding Corollary 3.2 in
  \cite{MR2353703}. The proof is the same for a general $\pi \in
  NC^*(1,m)$, but for completeness we still provide a proof.

  It is clear that $\Phi(\pi)(k) = \{k\}$ implies that $2k \sim_\pi
  2k-1$. Thus to prove the statement we have to prove that if $k$ and $l$
  are consecutive and distinct elements of a block of $\Phi(\pi)$ then $2k
  \sim_{\pi} 2l-1$.

  The first element in $\pi(2k)$ after $2k$ is odd, that is of the form
  $2p-1$, because $2k$ is even and the parity alternates in blocks of
  $\pi$. We claim that $p=l$. Note that we necessarily have $k < l \leq
  p$ (again for the cyclic order) because $k \sim_{\Phi(\pi)} p$.  Suppose
  that $k < l <p$. We get to a contradiction: indeed since $l
  \sim_{\Phi(\pi)} k$ and $\{2l-1,2l\} \subset \{2k+1,2k+2\dots,2p-2\}$
  there is at least one $j \in \{2k+1,2k+2\dots,2p-2\}$ and $i \in
  \{2p-1,2p\dots 2k\}$ such that $i \sim_\pi j$. But by definition of $p$,
  $j \nsim_\pi 2k$ and $j \nsim_\pi 2p-1$. This contradicts the fact that
  $\pi$ is non-crossing.
\end{proof}

We can now check that $\mathcal P$ is well-defined:
\begin{lemma}
\label{lemma=P_well_defined}
The map $\mathcal P$ from $NC^*(d,m)$ takes values in $NC(m)^{(d)}$.
\end{lemma}
\begin{proof}
  Let $\pi \in NC^*(d,m)$; we have to prove that if $1\leq i<j\leq d$ then
  $\Phi(\pi{\left|_{A_j}\right.})$ is finer than
  $\Phi(\pi{\left|_{A_i}\right.})$.

Let $\{k_1 < k_2 \dots < k_p\}$ be a block of
$\Phi(\pi{\left|_{A_i}\right.})$. Suppose that
$\Phi(\pi{\left|_{A_j}\right.})(k_1) \nsubseteq \{k_1, k_2 \dots k_p\}$.
Then there exist $1 \leq s \leq p$ and $l \notin \{k_1, k_2 \dots k_p\}$
such that $k_s$ and $l$ are consecutive elements of
$\Phi(\pi{\left|_{A_j}\right.})(k_1)$ (for the cyclic order). If $1 \leq t
\leq p$ is such that $k_t < l <k_{t+1}$ (with again the convention
$k_{p+1}=k_1$), we have by Lemma \ref{lemma=taille_du_quotient_pour_n=1}
that $2k_t \sim_{\pi{\left|_{A_i}\right.}} 2k_{t+1}-1$ and $2k_s
\sim_{\pi{\left|_{A_j}\right.}} 2l-1$, which contradicts the fact
that $\pi$ is non-crossing. This shows that
$\Phi(\pi{\left|_{A_j}\right.})(k_1) \subseteq\{k_1, k_2 \dots k_p\} =
\Phi(\pi{\left|_{A_i}\right.})(k_1)$. Since $k_1$ was arbitrary, the proof
is complete.
\end{proof}

Here is a last elementary lemma concerning general non-crossing partitions:
\begin{lemma} 
\label{lemma=number_of_succesives_in_partitions}
Let $N \in \N$ and $\pi \in NC(N)$ with $\alpha$ blocks. Then the number of
$k \in [N]$ such that $k \sim_\pi k+1$ is greater or equal to
$N-2(\alpha-1)$.
\end{lemma}
\begin{proof}
  For $\pi \in NC(N)$, let us denote by $c(\pi)$ the number of $k \in [N]$
  such that $k \sim_\pi k+1$. We prove by induction on $\alpha$ that if
  $\pi\in NC(N)$ has $\alpha$ blocks, then $c(\pi) \geq N-2(\alpha-1)$. If
  $\alpha = 1$, this is clear since $c(\pi)=N$.

  Assume that the statement of the lemma is true for all $N$ and all $\pi
  \in NC(N)$ with $\alpha$ blocks.  Take $\pi \in NC(N)$ with $\alpha+1$
  blocks. Since $\pi$ is non-crossing there is a block of $\pi$, say $A$,
  which is an interval of size $S$. If $\pi{\left|_{[N]\setminus
        A}\right.}$ is regarded as an element of $NC(N-S)$ then $c(\pi)
  \geq S-1 + c(\pi{\left|_{[N]\setminus A}\right.}) -1$. By the induction
  hypothesis $c(\pi{\left|_{[N]\setminus A}\right.}) \geq N-S-2(\alpha-1)$,
  which implies $c(\pi) \geq N-2\alpha$ and thus concludes the proof.
\end{proof}

The next Lemma is the main result of this section, and Theorem
\ref{thm=decomposition_of_NC*nm} will easily follow from it:
\begin{lemma}
\label{lemma=a_subset_of_great_size}
Let $\sigma \in NC^*_2(d,m)$. Then there is a subset $A$ of $[2dm]$ of size
greater than $2dm-4m$, which is a union of blocks of $\sigma$, and such
that for any $\pi \in NC^*(d,m)$ with $\mathcal P(\pi)= \mathcal P(\sigma)$
and any $k \in A$, $\pi(k) = \sigma(k)$.
\end{lemma}
\begin{proof}
  For any $1 \leq j \leq d$, denote by $\sigma_j =
  \Phi(\sigma{\left|_{A_j}\right.})$. Denote by $\sigma_{d+1} = 0_m$. Fix
  now $1 \leq i \leq d$ and $\{k_1<k_2<\dots <k_p\}$ a block of
  $\sigma_i$. As usual we take the convention that $k_{p+1}=k_1$. We claim
  that if $k_s \sim_{\sigma_{i+1}} k_{s+1}$ then for any $\pi \in
  NC^*(d,m)$ with $\mathcal P(\pi)= \mathcal P(\sigma)$, $\pi(2dk_s-i+1) =
  \{2dk_s-i+1,2dk_{s+1}-2d+i\}$ ($=\sigma(2dk_s-i+1)$ by Lemma
  \ref{lemma=taille_du_quotient_pour_n=1}).

  Let us first check that this claim implies the Lemma. By Lemma
  \ref{lemma=P_well_defined}, $\sigma_{i+1}$ is finer than $\sigma_i$ and
  in particular its restriction to $\{k_1,k_2,\dots ,k_p\}$ makes sense. By
  Lemma \ref{lemma=number_of_succesives_in_partitions}, the number of $s$'s
  in $\{1,\dots ,p\}$ such that $k_s \sim_{\sigma_{i+1}} k_{s+1}$ is greater
  than $p-2(|{\sigma_{i+1}}{\left|_{\{k_1,k_2,\dots ,k_p\}}\right.}| -1)$
  where $|\sigma|$ is the number of blocks of $\sigma$. Thus summing over
  all blocks of $\sigma_i$ we get at least $2m-4(|\sigma_{i+1}| -
  |\sigma_i|)$ elements $k$ in $A_i$ such that $\pi(k) = \sigma(k)$ for any
  $\pi \in NC^*(d,m)$ with $\mathcal P(\pi)= \mathcal P(\sigma)$. This
  allows to conclude the proof since
\[\sum_{i=1}^d \left(2m-4(|\sigma_{i+1}| - |\sigma_i|)\right) = 2md -4
|\sigma_{d+1}| + 4 |\sigma_1| >  2md -4m.\]

Note that $A$ is constructed as a union of blocks of $\sigma$. 

We now only have to prove the claim. Assume that $k_s \sim_{\sigma_{i+1}}
k_{s+1}$ and take $\pi \in NC^*(d,m)$ such that $\mathcal P(\pi) = \mathcal
P(\sigma)$. By Lemma \ref{lemma=taille_du_quotient_pour_n=1} applied to
$\Phi(\sigma{\left|_{A_i}\right.}) = \sigma_i$, $2dk_s-i+1 \sim_\pi
2dk_{s+1}-2d+i$. Thus we only have to prove that if $k_s
\sim_{\sigma_{i+1}} k_{s+1}$ there is no $k \in \{k_1,\dots , k_p\}
\setminus \{k_{s+1}\}$ such that $2dk_s-i+1 \sim_\pi 2dk-2d+i$.

But if $k_s \sim_{\sigma_{i+1}} k_{s+1}$ then $i \neq d$ (because
$\sigma_{d+1} = 0_{m}$) and by Lemma \ref{lemma=P_well_defined}, $k_s$ and
$k_{s+1}$ are consecutive elements in $\sigma_{i+1}(k_s)$. Thus by Lemma
\ref{lemma=taille_du_quotient_pour_n=1}, $2dk_s-i \sim_\pi
2dk_{s+1}-2d+i+1$. The condition that $\pi$ is non-crossing implies the
claim since for $k \in \{k_1,\dots ,k_p\} \setminus \{k_{s+1}\}$,
\[ 2dk_s-i< 2dk_s-i+1 < 2dk_{s+1}-2d+i+1 <2dk-2d+i,\]
that is $(2dk_s-i+1,2dk-2d+i)$ and $(2dk_s-i,2dk_{s+1}-2d+i+1)$ are
crossing.

\end{proof}

 We can now prove the Theorem.
\begin{proof}[Proof of Theorem \ref{thm=decomposition_of_NC*nm}]
  Let $\sigma \in NC^*_2(d,m)$. If $\pi \in NC^*(d,m)$ satisfies $\mathcal
  P(\pi)=\sigma$ then Lemma \ref{lemma=taille_du_quotient_pour_n=1} applied
  to $\sigma{\left|_{A_i}\right.}$ and $\pi{\left|_{A_i}\right.}$ for
  $i=1,\dots,d$ proves that $\sigma$ is finer than $\pi$, and
  \ref{lemma=a_subset_of_great_size} implies that $\pi$ has at least
  $dm-2m$ blocks of size $2$. The fact that every block of $\pi$ has size
  at most $m$ just follows from the definition of $NC^*(d,m)$: $\pi$ is
  indeed finer than $\{A_1,\dots A_d\}$ with $|A_j| =2m$.

We now prove the first statement of Theorem
\ref{thm=decomposition_of_NC*nm}. Let $A$ be the subset of $[2dm]$ given by
Lemma \ref{lemma=a_subset_of_great_size}.  Then there is an injection:
\[ \begin{array}{ccl}\left\{ \pi \in NC^*(d,m), \mathcal P(\pi)=\mathcal
    P(\sigma) \right\} & \to & NC([2dm]\setminus A) \\
  \pi & \mapsto & \pi{\left|_{[2dm]\setminus A}\right.}
\end{array}\]
In particular since there are less than $4^N$ non-crossing partitions on
$[N]$, the first statement of the Theorem follows with $4^{2m}$ replaced by
$4^{4m}$ because $[2dm]\setminus A$ has cardinality less than $4m$. To get
the $4^{2m}$ just replace $[2dm]\setminus A$ by a set $B$ that contains
exactly one element of $\sigma(k)$ for any $k \in [2dm]\setminus A$. Then
$B$ has cardinality less than $2m$ because $[2dm] \setminus A$ is a union
of blocks (=pairs) of $\sigma$, and the previous map is still an
injection since $\pi \in NC^*(d,m)$ and $\mathcal P(\pi)=\mathcal
P(\sigma)$ implies that $\sigma$ is finer that $\pi$.
\end{proof}

\subsection{Invariant of the $P_k$'s}
\label{part=invariants_of_P_k}
Motivated by Lemma \ref{lemma=CauchySchwarz} we are interested in
invariants of the operations $P_{kd}$ on $NC^*(d,m)$. For $\pi \in
NC^*(1,m)$ let $B(\pi)$ be the number of blocks in $\Phi(\pi)$. This is the
fundamental observation:
\begin{lemma}
\label{lemma=combinatorial_invariant}
For any $\pi \in NC^*(1,m)$,
\[B(\pi) = \frac 1 2 \left(B(P_k(\pi)) + B(P_{k+m}(\pi))\right).\]
\end{lemma}

This Lemma is a consequence of the following description, which proves that
for any $k$, the set of blocks of $\Phi(\pi)$ but one is in bijection with
the set of blocks of $\pi$ that do not contain $k$ and that begin with an
odd element (after $k$ for the cyclic order):
\begin{lemma}
\label{lemma=B_is_the_number_of_first_odd}
Let $k \in [2m]$ and $\pi \in NC^*(1,m)$. Then $B(\pi)-1$ is equal to the
number of $l \in [2m] \setminus \{k\}$ such that $l$ is odd and such that
for any $l' \sim_\pi l$, $ l \leq l' <k$ (for the cyclic order). 
\end{lemma}
\begin{proof}
  Indeed the set of odd $l$'s different from $k$ such that $l' \sim_\pi l
  \Rightarrow l \leq l' <k$ (for the cyclic order) is in bijection with the
  blocks of $\Phi(\pi)$ that do not contain $\lfloor{(k+1)/2} \rfloor$.

  The direct map consists in mapping to any such $l$ the block
  $\Phi(\pi)(\lfloor{(l+1)/2} \rfloor)$ and the reverse map gives to any
  block $A$ of $\Phi(\pi)$ no containing $\lfloor{(k+1)/2} \rfloor$ the
  smallest $l$ greater than $k$ (again for the cyclic order) such that
  $\lfloor{(l+1)/2} \rfloor \in A$. The reader can check using Lemma
  \ref{lemma=taille_du_quotient_pour_n=1} that these maps are indeed
  inverses of each other.
\end{proof}

\begin{proof}[Proof of Lemma \ref{lemma=combinatorial_invariant}]
  We use Lemma \ref{lemma=B_is_the_number_of_first_odd} with $k+1$ instead
  of $k$. For any $\pi \in NC^*(1,m)$ we denote by $F(\pi,k)$ the set of
  odd $l \in [2m] \setminus \{k+1\}$ such that $l' \sim_{\pi} l
  \Rightarrow l \leq l' <k+1$. We know that $|F(\pi,k)| =
  B(\pi)-1$. Moreover let us decompose $F(\pi,k)$ as the disjoint union of
  $F_1(\pi,k)$ and $F_2(\pi,k)$ defined by: $l \in F_1(\pi,k)$ if and only
  $l \in F(\pi,k)$ and $\pi(l) \subset I_{k+m}$; and $F_2(\pi,k)$ is the
  set of $l \in F(\pi,k)$ such that $\pi(l) \cap I_l \neq \emptyset$.

  If $l \in I_{k+m}$ then $l \in F(P_{k+m}(\pi),k)$ if and only if $l \in
  F(\pi,k)$ because if $k+1 \leq l' <l$, then $l' \sim_{P_{k+m}(\pi)} l$ if
  and only if $l' \sim_{\pi} l$.
  
  Take now $l \notin I_{k+m}$. By definition of $F(\cdot,k)$, $l$ is in
  $F(P_{k+m}(\pi),k)$ if and only if $l$ is odd and $l$ is the first
  element (after $k+1$ for the cyclic order) of a block of $P_{k+m}(\pi)$
  contained in $I_k$, which is equivalent to the fact that $s_k(l) =
  2k+1-l$ is even and is the last element of a block of $\pi$ contained in
  $I_{k+m}$. Such a block then has first element odd, and thus belongs to
  $F_1(\pi,k)$ except if it is equal to $k+1$. To summarize, we have thus
  proved that
\begin{equation} 
\label{eq=value_of_Fpk+m_of_pi1}
|F(P_{k+m}(\pi),k)| = |F(\pi,k) \cap I_{k+m}| +
  |F_1(\pi,k)| +1 \end{equation}

if $k+1$ is odd and $\pi(k+1) \subset I_{k+m} = \{k+1,k+2,\dots ,k+m\}$, and 
\begin{equation}
\label{eq=value_of_Fpk+m_of_pi2}
 |F(P_{k+m}(\pi),k)| = |F(\pi,k) \cap I_{k+m}| +
  |F_1(\pi,k)| \end{equation} otherwise.

We now compute $|F(P_k(\pi),k)|$. If $l \in I_k$ then as above $l \in
F(P_{k}(\pi),k)$ if and only if $l \in F(\pi,k)$. If $l \notin I_k$
then $l \in F(P_{k}(\pi),k)$ if and only if $l$ is odd and $l$ is the
first element strictly after $k+1$ (in the cyclic order) of a block of
$P_k(\pi)$ not containing $k+1$. By construction of $P_k(\pi)$ this is
equivalent to the fact that $s_k(l) = 2k+1-l$ is even, belongs to
$I_k$, is different from $k$ and is the last element before $k$ in a
block of $\pi$. The first element (strictly after $k$ in the cyclic
order) of such a block is then in $F_2(\pi,k)$ except if it is equal
to $k+1$. Reciprocally, if $l'$ is the last element of a block
containing an element of $F_2(\pi,k)$ then $l = s_k(l') \in
F(P_k(\pi),k)$ except if $l'=k$. The same is true if $\pi(k+1)
\nsubseteq I_{k+m}$, $k+1$ is odd and if $l'$ denotes the last element
in $\pi(k+1)$. Thus
\begin{eqnarray*}
  |F(P_k(\pi),k)| & = & |F(\pi,k)\cap I_k| + |F_2(\pi,k)| - 1_{k\textrm{ is
      even}} + 1_{k\textrm{ is even and } \pi(k+1)\nsubseteq I_{k+m}}\\
  & = & |F(\pi,k)\cap I_k| + |F_2(\pi,k)| - 1_{k\textrm{ is
      even and } \pi(k+1)\subset I_{k+m}}. 
\end{eqnarray*}

Summing this last equality with \eqref{eq=value_of_Fpk+m_of_pi1} or
\eqref{eq=value_of_Fpk+m_of_pi2} yields
\begin{multline*}
 |F(P_k(\pi),k)| + |F(P_{k+m}(\pi),k)| = \\ |F(\pi,k)\cap I_k| +
|F_2(\pi,k)| + |F(\pi,k)\cap I_{k+m}| + |F_1(\pi,k)|
 = 2 |F(\pi,k)|.
\end{multline*}

This concludes the proof since by Lemma
\ref{lemma=B_is_the_number_of_first_odd} for any $\sigma \in NC^*(1,m)$,
$|F(\sigma,k)| = B(\sigma)-1$.
\end{proof}

\subsection{Study of $NC(d,m)$}\label{part=NCdmsansstar}
Another relevant subset of $NC(2dm)$ is the set $NC(d,m)$ of partitions
$\pi$ with blocks of even cardinality and that connect only elements of
different intervals $J_k$. In other words for all $i,j \in [2dm]$, $i
\nsim_\pi j$ if $i,j \in J_k$.

The following observation is very simple but, in view of Theorem
\ref{thm=non_holomo_Rdiag} or \ref{thm=non_holomo_selfadj}, it is the
motivation for the introduction of $NC(d,m)$ :
\begin{lemma} 
  \label{lemma=characterization_of_NCdm} Let $\pi \in NC(2dm)$ with blocks
  of even cardinality. Then $\pi \in NC(d,m)$ if and only if $\pi$ does not
  connect two consecutive elements of a same subinterval $J_i$. In other
  words, $i\sim_\pi i+1$ only if $i$ is a multiple of $d$.
\end{lemma}
\begin{proof}
  The \emph{only if} part of the proof is obvious. The converse follows
  from the fact that a non-crossing partition always contains an interval
  (if $\pi$ is non-crossing with blocks of even size, and $s<t \in J_i$
  with $s \sim_\pi t$ and $t\neq s+1$, apply this fact to
  $\pi{\left|_{\{s,s+1,\dots,t-1\}}\right.}$).
\end{proof}

The purpose of this section is to generalize Theorem
\ref{thm=decomposition_of_NC*nm}. Namely we prove
\begin{thm}
\label{thm=decomposition_of_NCn_m}
The cardinality of $NC(d,m)$ is less than $(4d+4)^{2m}$.

Moreover for any $\pi \in NC(d,m)$ the number of blocks of $\pi$ of size
$2$ is greater than $(d-2)m$.
\end{thm}
The idea of the proof is similar to the proof of Theorem
\ref{thm=decomposition_of_NC*nm}: we try to reduce to the subset of
$NC(d,m)$ consisting of partitions into pairs. For this we introduce the
map $Q = Q^{(N)}$ from the set of non-crossing partitions of $[2N]$ into
blocks of even sizes to the set of non-crossing partitions of $[2N]$ into
pairs. The map $Q$ has the property that if $\pi \in NC(2N)$ has blocks of
even sizes, then $Q(\pi)$ is finer than $\pi$ and any block $\{k_1,\dots
,k_{2p}\}$ of $\pi$ with $1 \leq k_1<\dots <k_{2p} \leq 2N$ becomes $p$
blocks of $Q(\pi)$, namely $\{k_1,k_2\},\dots,\{k_{2p-1},k_{2p}\}$. It is
straightforward to check that this indeed defines a non-crossing partition
of $[2N]$ into pairs.  Note that unlike in the rest of the paper here the
element $1 \in [2N]$ plays a specific role in the definition of $Q$ and we
abandon the cyclic symmetry of $[2N]$. But this has the advantage to allow
to define an order relation on the set of pairs of elements of $[2N]$: we
will say that a pair $(i,j)$ covers a pair $(k,l)$ if $1 \leq i<k<l<j\leq
2N$.

A noteworthy property of $Q$ is that if $\sigma = Q(\pi)$ then two blocks
(=pairs) of $\sigma$ cannot be contained in the same block of $\pi$ if one
covers the other. In other words if $1\leq i <k<l <j\leq 2N$ with 
$i\sim_\sigma j$ and $k\sim_\sigma l$ then $i \nsim_\pi k$.

Following the notation of section 3.1 in \cite{MR2353703}, the image
$Q(NC(d,m))$ is denoted by $\mathscr I(d,m)$; it is the set of partitions
of $\pi$ into pairs that do not connect elements of a same subinterval
$J_k$ for $k=1,\dots,2m$. We are not aware of any nice combinatorial
description of $\mathscr I(d,m)$ as for $NC^*_2(d,m)$, but a precise bound
for its cardinality is known: by the proof of Theorem 5.3.4 in
\cite{MR1660906}, the cardinality of $\mathscr I(d,m)$ is equal to
$\tau(T_d(s)^{2m})$ where $T_d$ is the $d$-th Tchebycheff polynomial and $s$
is a semicircular element of variance $1$ in a tracial $C^*$-algebra
$(A,\tau)$. In particular since $\|T_d(s)\| = d+1$ we have that $|\mathscr
I(d,m)|\leq (d+1)^{2m}$. Theorem \ref{thm=decomposition_of_NCn_m} will thus
follow from the following more general statement:

\begin{lemma}
\label{thm=description_of_NCpart_withKintervals}
  Suppose that $[2N]$ is divided into $k$ non-empty intervals $S_1,\dots,S_k$ and let
  $\sigma$ be a non-crossing partition of $[2N]$ into pairs that do not
  connect elements of a same subinterval $S_i$. Then there are at most
  $4^{k-2}$ non-crossing partitions $\pi$ of $[2N]$ that do not connect
  elements of a same subinterval $S_i$ and such that
  $Q(\pi)=\sigma$. Moreover for such a $\pi$ there are at most $2k-4$
  elements $i \in [2N]$ for which $\pi(i)$ is not a pair.
\end{lemma}
\begin{proof} We prove this statement by induction on $N$. For simplicity
  of notation we will assume that the intervals $S_1,\dots,S_k$ are
  ordered, \emph{i.e.}  that if $i \in S_s$ and $j\in S_t$ with $s<t$ then
  $i<j$.

If $N=1$ and $\sigma$ is as above then $\sigma = 1_2$, $k=2$, and there is
only one $\pi \in NC(2)$ with $Q(\pi)=\sigma$. This proves the assertion
for $N=1$.

Assume that the above statement holds for $1,2,\dots,N-1$ and take $\sigma$
as above. Consider the set $\left\{\{s_i,t_i\},i=1\dots p\right\}$ of
outermost blocks (=pairs) of $\sigma$, \emph{i.e} the set of pairs of
$\sigma$ that are not being covered by another block of $\sigma$.
If we order the $s_i$'s and $t_i$'s so that $s_i<t_i$ and $s_i < s_{i+1}$
then we have that $s_1=1$, $s_{i+1}=t_i+1$ and $t_p=2N$.

By the property of $Q$ mentioned above, a partition $\pi \in NC(2N)$ that
does not connect elements of the same interval $S_j$ (for $j=1,\dots,k$)
satisfies $Q(\pi)=\sigma$ if and only if the following properties are
satisfied:
\begin{itemize}
\item For any $1 \leq i \leq p$, $\{s_i+1,\dots, t_i-1\}$ is a union of
  blocks of $\pi$, the non-crossing partition $\pi{\left|_{\{s_i+1,\dots
        t_i-1\}}\right.}$ does not connect elements of the same subinterval
  $S_j\cap \{s_i+1,\dots t_i-1\}$ for $j=1,\dots,k$, and
  $Q(\pi{\left|_{\{s_i+1,\dots t_i-1\}}\right.}) =
  \sigma{\left|_{\{s_i+1,\dots t_i-1\}}\right.}$.
\item Any block of $\pi{\left|_{\{s_1,t_1,s_2,t_2,\dots,s_p,t_p\}}\right.}$
  is a union of pairs $\{s_i,t_i\}$ and does not contain $2$ elements of a
  same interval $S_j$.
\end{itemize}
Define $k_+(i)$ and $k_-(i)$ for $1 \leq i \leq p$ by $s_i \in S_{k_-(i)}$
and $t_i \in S_{k_+(i)}$. Then for any $1\leq i \leq p$, $k_-(i)<k_+(i)$ and
for $i<p$, $k_+(i) \leq k_-(i+1)$. 

Since $\{s_i+1,\dots t_i-1\}$ intersects at most $k_+(i)-k_-(i)+1$
different intervals $S_j$, we have by the induction hypothesis that the number
of non-crossing partitions of $\{s_i+1,\dots,t_i-1\}$ that satisfy the
first point above is at most $4^{k_+(i)-k_-(i)-1}$, and for such a
partition at most $2(k_+(i)-k_-(i)-1)$ elements of $\{s_i+1,\dots t_i-1\}$
do not belong to a pair. 

Moreover the set of non-crossing partitions of
$\{s_1,t_1,s_2,t_2,\dots,s_p,t_p\}$ that satisfy the second point is in
bijection with the set of non-crossing partitions of $\{s_i,i=1\dots p\}$
such that $s_i \nsim s_{i+1}$ if $k_+(i) = k_-(i+1)$. Its cardinality is in
particular less than (or equals) the number of non-crossing partitions of
$[p]$, which is less than $4^{p-1}$. Therefore the total number of
non-crossing partitions $\pi$ of $[2N]$ that do not connect elements of a
same subinterval $S_j$ and such that $Q(\pi)=\sigma$ is less than
\[4^{p-1} \prod_{i=1}^p 4^{k_+(i)-k_-(i)-1} \leq 4^{k-2}.\] 
We used the inequality $\sum_{i=1}^p k_+(i)-k_-(i)-1 \leq k-1-p$. 

To prove that for such a $\pi$ at most $2k-4$ elements of $[2N]$ do
not belong to a pair of $\pi$, note that for an element $j\in[2N]$ the
block $\pi(j)$ is not a pair either if $j
\in\{s_1,t_1,\dots,s_p,t_p\}$ or if $j$ belongs to a block of
$\pi{\left|_{\{s_i+1,\dots t_i-1\}}\right.}$ which is not a pair for
some $1\leq i\leq p$. If $k_+(i)<k_-(i+1)$ for some $i$ then we are
done since $2p+\sum_{i=1}^p 2k_+(i)-2k_-(i)-2 \leq 2k-4$. To conclude
the proof we thus have to check that if $k_+(i)=k_-(i+1)$ for any
$1\leq i <p$ then there are at least $2$ elements of
$\{s_1,t_1,\dots,s_p,t_p\}$ that belong to a pair of $\pi$. But this
amounts to showing that a non-crossing partition of $[p]$ such that $i
\nsim i+1$ for any $1 \leq i <p$ contains at least one singleton,
which is clear.
\end{proof}

The following Lemma is also an easy extention of Lemma
\ref{lemma=symmetrization}. Remember that the partitions $\sigma_l^{(d,m)}$
and $\widetilde \sigma_l^{(d,m)}$ are defined in Corollary
\ref{corollary=symmetrization_for_n}:
\begin{lemma} \label{lemma=symmetrization_forNCdm}Fix integers $d$ and $m$.

  For any $k \in [2m]$ and $\pi \in NC(d,m)$ the partition $P_{kd}(\pi)$
  also belongs to $NC(d,m)$.

  Let $k \in \N$ such that $2^{k} \geq m$. Then for any partition $\pi \in
  NC(d,m)$, the partition $\pi_k = P_{2^k} P_{2^{k-1}} \dots P_{2} P_{1}
  P_{m} (\pi)$ is one of the $2d+1$ partitions $\sigma_l^{(d,m)}$ for $0
  \leq l \leq d$ or $\widetilde \sigma_l^{(d,m)}$ for $1 \leq l \leq d$.
\end{lemma}
\begin{proof} The first point is straightforward.

  The proof of the second point is the same as Lemma
  \ref{lemma=symmetrization}: depending on the fact that
  $\{1,2,\dots,dm\}\cap \pi(i) \setminus \{i\}$ and $\{dm+1,\dots,2dm\}\cap
  \pi(i)$ are empty or not for $i=1,\dots ,d$, we prove by induction on $k$
  that $\pi_k$ has the right properties. The details are left to the
  reader.
\end{proof}

\section{Inequalities}
For any partition $\pi$ of $[2N]$, and any $k=(k_1,\dots ,k_{2N})\in
\N^{2N}$, we write $k \respecte \pi$ if for any $i,j \in [2N]$ such that $i
\sim_\pi j$, $k_i=k_j$.

Let $a=(a_k)_{k \in \N^N}$ be a finitely supported family of matrices. For
any $k=(k_1,\dots ,k_N) \in \N^N$ let $\widetilde a_k =
a_{(k_N,k_{N-1},\dots ,k_1)}$.

For such $a$ and for a partition $\pi$ of $[2N]$, we denote by
$S(a,\pi,N,1)$ the following quantity:
\begin{equation}
\label{eq=def_of_SapiN1}
S(a,\pi,N,1) = \sum_{k,l \in \N^N, (k,l) \respecte \pi} Tr(a_{k} \widetilde
a_l^*).
\end{equation}

More generally for integers $m,d$, for a finitely supported family of
matrices $a= (a_k)_{k \in \N^d}$ and a partition $\pi$ of $[2dm]$, we
define
\begin{equation}
\label{eq=def_of_S_apidm}
  S(a,\pi,d,m) = \sum_{k_1,\dots ,k_{2m} \in \N^d, (k_1,\dots, k_{2m})
    \respecte \pi} Tr(a_{k_1} \widetilde a_{k_2}^* a_{k_3} \dots
    a_{k_{2m-1}} \widetilde a_{k_{2m}}^* ).
\end{equation}
In this equation and in the rest of the paper an element $k=(k_1,\dots,
k_{2m}) \in (\N^d)^{2m}$ is identified with an element of
$\N^{2dm}$. Therefore the expression $k \respecte \pi$ has a meaning for
$\pi \in NC(2dm)$.

The following application of the Cauchy-Schwarz inequality is what
motivates the introduction of the operations $P_k$ on the partitions of
$[2N]$. The same use of the Cauchy-Schwarz inequality has been made in the
second part of \cite{MR1812816}.
\begin{lemma}\label{lemma=CauchySchwarz}
  For a partition $\pi$ of $[2N]$ and a finitely supported family of
  matrices $a=(a_k)_{k \in \N^N}$,
\[ \left|S(a,\pi,N,1)\right| \leq \left(S(a,P_0(\pi),N,1)\right)^{1/2}
\left(S(a,P_N(\pi),N,1)\right)^{1/2}.\]

More generally for a partition $\pi$ of $[2dm]$, for a finitely supported
family of matrices $a= (a_k)_{k \in \N^d}$ and any integer $i$
\begin{equation}
\label{eq=CauchySchwarz}
 \left|S(a,\pi,d,m)\right| \leq \left(S(a,P_{di}(\pi),d,m)\right)^{1/2}
\left(S(a,P_{(m+i)d}(\pi),d,m)\right)^{1/2}.\end{equation}
\end{lemma}
\begin{proof}
  The second statement for $i=0$ follows from the first one by replacing
  $N$ by $dm$. Indeed for any and $k=(k_1,\dots,k_m) \in (\N^d)^m \simeq
  \N^{dm}$, denote $\beta_k = a_{k_1} \widetilde a_{k_2}^* a_{k_3} \dots
  a_{k_m}$ if $m$ is odd and $\beta_k=a_{k_1} \widetilde a_{k_2}^* a_{k_3}
  \dots \widetilde a_{k_m}^*$ if $m$ is even. We claim that $S(a,\pi,d,m) =
  S(\beta,\pi,dm,1)$. We give a proof when $m$ is odd, the case when $m$ is
  even is similar. It is enough to prove that if $k=(k_1,\dots,k_m) \in
  (\N^d)^m$ then $\widetilde \beta_k^* = \widetilde a_{k_1}^* a_{k_2} \dots
  \widetilde a_{k_m}^*$. But if $r:\N^d\to\N^d$ denotes the map
  $r(s_1,\dots,s_d)=(s_d,\dots,s_1)$ we have that
\begin{eqnarray*}
  \widetilde \beta_k^* & = & \beta_{r(k_m),\dots,r(k_1)}^*= \big(a_{r(k_m)}
  \dots \widetilde a_{r(k_2)}^* a_{r(k_1)}\big)^*\\
  & = & a_{r(k_1)}^* \widetilde a_{r(k_2)} \dots a_{r(k_m)}^*\\
  & = & \widetilde a_{k_1}^* a_{k_2} \dots
  \widetilde a_{k_m}^*.
\end{eqnarray*}

For a general $i$ the following argument based on the trace property allow
to reduce to the case $i=0$: for a partition $\pi$ of $[2dm]$ and any $n
\in [2dm]$ denote $\tau_{n}(\pi)$ the partition such that $s
\sim_{\tau_{n}(\pi)} t$ if and only if $s+n \sim_\pi t+n$, so that
$P_{n+k}(\pi) = (\tau_{n}^{-1} \circ P_k \circ \tau_n)(\pi)$ for any
integer $k$. Moreover by the trace property $S(a,\pi,d,m) =
S(a,\tau_{di}(\pi),d,m)$ if $n$ is even and $S(a,\pi,d,m) = S(\widetilde
a^*,\tau_{di}(\pi),d,m)$ if $i$ is even (here $\widetilde a^*$ denotes the
family $(\widetilde a_k^*)_{k \in \N^d}$). Therefore if one assumes that
the inequality \eqref{eq=CauchySchwarz} is satisfied for any $\pi$ and any
$a$ but only for $i=0$, then we can deduce it for a general $i$ in the
following way. Denote $b = (a_k)_{k \in \N^d}$ if $i$ is even and $b =
(\widetilde a_k^*)_{k \in \N^d}$ if $i$ is odd and :
\begin{eqnarray*}
|S(a,\pi,d,m)|^2 &=& |S(b,\tau_{di}(\pi),d,m)|^2\\
& \leq & S(b,P_0(\tau_{di}(\pi)),d,m) S(b,P_{dm}(\tau_{di}(\pi)),d,m)\\
& = & S(b,\tau_{di}(P_{di}(\pi)),d,m) S(b,\tau_{di}(P_{dm+di}(\pi)),d,m)\\
&=& S(a,P_{di}(\pi),d,m) S(a,P_{(m+i)d}(\pi),d,m)
\end{eqnarray*}

  We now prove the first statement. We take the same notation as in
  Definition \ref{def=intervals_symmetries}.

  Let us clarify the notation for the rest of the proof. In the whole
  proof, for a set $X$ we see a $k \in \N^X$ as a function from $X$ to
  $\N$, and for an integer $N$ we will identify $\N^N$ with $\N^{[N]}$. In
  particular, if $X$ and $Y$ are disjoint subsets of a set $Z$, and if
  $k\in \N^X$ and $l \in \N^Y$, $[k,l]$ will denote the element of $\N^{X
    \cup Y}$ corresponding to the function on $X \cup Y$ that has $k$ as
  restriction to $X$ and $l$ as restriction to $Y$.
  
  Let us denote by $A$ the union of the blocks of $\pi$ that are contained
  in $I_N=\{1,\dots N\}$, by $B$ the union of the blocks of $\pi$ that are
  contained in $[2N] \setminus I_N = \{N+1,\dots, 2N\} = I_{2N}$ and by $C$
  the rest of $[2N]$. In the following equations, $s$ will vary in $\N^A$,
  $t$ in $\N^{I_N\setminus A}$, $u$ in $\N^B$ and $v$ in $\N^{I_{2N}
    \setminus B}$. For such $s$,$t$,$u$ and $v$ and with the previous
  notation, $[s,t,u,v] \respecte \pi$ if and only if $s \respecte
  \pi{\left|_A\right.}$, $[t,v] \respecte \pi{\left|_C\right.}$ and $u
  \respecte \pi{\left|_B\right.}$. For $k\in \N^{I_{2N}}$ (\emph{i.e.} $k$
  is a function $k:I_{2N}\to \N$), we will also abusively denote
  $\widetilde a_k \egdef \widetilde a_{(k(N+1),\dots,k(2N))}$. With this
  notation the definition in \eqref{eq=def_of_SapiN1} becomes
\begin{eqnarray*}
S(a,\pi,N,1) &=& \sum_{\begin{array}{c}s \in \N^A,t\in \N^{I_N \setminus A},u
  \in \N^B, v \in \N^{I_{2N} \setminus B}\\ {[}s,t,u,v{]} \respecte \pi
\end{array}} Tr(a_{[s,t]} \widetilde a_{[u,v]}^*)\\
& = & \sum_{\begin{array}{c} t, v  \\ {[}t,v{]}\respecte \pi{\left|_C\right.}
      \end{array}} 
Tr\left( \big(\sum_{s \respecte \pi{\left|_A\right.}}a_{[s,t]} \big) 
\big(\sum_{u \respecte \pi{\left|_B\right.}}\widetilde a_{[u,v]}\big)^*
\right).
\end{eqnarray*}
Thus 
\[\left|S(a,\pi,N,1) \right| \leq \sum_{{[}t,v{]}\respecte
  \pi{\left|_C\right.}} \left|Tr\left( \big(\sum_{s \respecte
\pi{\left|_A\right.}}a_{[s,t]} \big) \big(\sum_{u \respecte
\pi{\left|_B\right.}}\widetilde a_{[u,v]}\big)^* \right)  \right|.\]
Applying the Cauchy-Schwarz inequality for the trace, we get
\[ \left|S(a,\pi,N,1) \right| \leq \sum_{{[}t,v{]}\respecte
  \pi{\left|_C\right.}} \left\| \sum_{s \respecte
\pi{\left|_A\right.}}a_{[s,t]} \right\|_2  \left\|\sum_{u \respecte
\pi{\left|_B\right.}}\widetilde a_{[u,v]} \right\|_2.  \]
The classical Cauchy-Schwarz inequality yields
\[\left|S(a,\pi,N,1) \right| \leq (1)^{1/2} (2)^{1/2} \]
where 
\begin{eqnarray*}
(1) &=& \sum_{{[}t,v{]}\respecte \pi{\left|_C\right.}} \left\| \sum_{s
  \respecte \pi{\left|_A\right.}}a_{[s,t]} \right\|_2^2\\
(2) & = & \sum_{{[}t,v{]}\respecte \pi{\left|_C\right.}} \left\| \sum_{u
  \respecte \pi{\left|_B\right.}}\widetilde a_{[u,v]} \right\|_2^2. 
\end{eqnarray*}
We claim that $(1) = S(a,P_N(\pi),N,1)$ and $(2) = S(a,P_0(\pi),N,1)$. We
only prove the first equality, the second is proved similarly (or follows
from the first). But
\begin{eqnarray*}
(1) &=& \sum_{{[}t,v{]}\respecte \pi{\left|_C\right.}} \left\| \sum_{s
  \respecte \pi{\left|_A\right.}}a_{[s,t]} \right\|_2^2\\
&=& \sum_{{[}t,v{]}\respecte \pi{\left|_C\right.}} Tr\left( \big(\sum_{s \respecte
\pi{\left|_A\right.}}a_{[s,t]} \big) \cdot \big(\sum_{s \respecte
\pi{\left|_A\right.}}a_{[s,t]} \big)^*\right)\\
&=&  Tr\left( \sum_{{[}t,v{]}\respecte \pi{\left|_C\right.}} \sum_{s \respecte
\pi{\left|_A\right.}} \sum_{s' \respecte
\pi{\left|_A\right.}}a_{[s,t]} a_{[s',t]}^* \right)\\
&=&  Tr\left( \sum_{{[}t,v{]}\respecte \pi{\left|_C\right.}} \sum_{s \respecte
\pi{\left|_A\right.}} \sum_{s' \respecte
\pi{\left|_A\right.}}a_{[s,t]} \widetilde a_{r([s',t])}^* \right),
\end{eqnarray*}
where on the last line for any $k = (k_{1},\dots, k_{N}) \in
\N^{I_{N}}$, $r(k) \in \N^{I_{N}}$ is defined by $r(k) = (k_{N},
k_{N-1} ,\dots, k_{1})$. 

By definition of $B$, for any $j \in I_{2N} \setminus B$ there is
$i \in I_N \setminus A$ such that $i \sim_\pi j$. Thus for any $t \in
\N^{I_N \setminus A}$ there is exactly one or zero $v \in\N^{I_{[2N]}
\setminus B}$ such that ${[}t,v{]} \respecte \pi_C$, depending whether $t
\respecte \pi_{I_N \setminus A}$ or not.

The claim that $(1) = S(a,P_N(\pi),N,1)$ thus follows from the observation
that for $k,l \in \N^N$, $(k,l) \respecte P_N(\pi)$ if and only there are
$s,s'\in \N^A$ and $t \in \N^{I_N \setminus A}$ such that $k=[s,t]$,
$l=r([s',t])$ and $s\respecte \pi{\left|_A\right.}$, $s'\respecte
\pi{\left|_A\right.}$ and $t\respecte \pi{\left|_{I_N \setminus
      A}\right.}$.
\end{proof}

We now have to observe that the quantities $S(a,\sigma_l^{(d,m)},d,m)$ for
$l=0,\dots, d$ and $S(a, \widetilde \sigma_l^{(d,m)},d,m)$ for $l=0,\dots, d$
have simple expressions.

A (finitely supported) family of matrices $a=(a_k)_{k \in \N^d}$ can be
made in various natural ways into a bigger matrix, for any decomposition of
$\N^d \simeq \N^l \times \N^{d-l}$.  If the $a_k$'s are viewed as operators
on a Hilbert space $H$ ($H=\C^\alpha$ if the $a_k$'s are in
$M_\alpha(\C)$), then let us denote by $M_l$ the operator from $H \otimes
\ell^2(\N)^{\otimes {d-l}}$ to $H \otimes \ell^2(\N)^{\otimes l}$ having
the following block-matrix decomposition:
\[\left(a_{[s,t]}\right)_{s\in \N^{\{1,\dots, l\}},t\in \N^{\{l+1,\dots,
 d\}}}.\]

Note that since $(a_k)$ has finite support, the above matrix has only
finitely many nonzero entries, and hence corresponds to a finite rank
operator. In particular, it belongs to $S_p\left (H \otimes \ell^2(\N)^{\otimes
{d-l}}; H \otimes \ell^2(\N)^{\otimes l}\right)$ for any $p \in (0,\infty]$.

\begin{lemma} 
\label{lemma=identifications} 
  Let $d$, $m$, $a=(a_k)_{k \in \N^d}$ and $M_l$ as above, and $\sigma_l$
  and $\widetilde \sigma_l$ defined in Corollary
  \ref{corollary=symmetrization_for_n}. Then for $l\in\{0,1,\dots, d\}$:
    \[ S(a,\sigma_l^{(d,m)},d,m) = \left\|M_l\right\|_{S_{2m}\left
        (H \otimes \ell^2(\N)^{\otimes {d-l}}; H \otimes \ell^2(\N)^{\otimes l}\right) }^{2m}.\]
Moreover for $l \in \{1,\dots, d\}$
\[ S(a,\widetilde \sigma_l^{(d,m)},d,m) \leq
\left\|M_l\right\|_{S_{2m}\left (H \otimes \ell^2(\N)^{\otimes {d-l}}; H
    \otimes \ell^2(\N)^{\otimes l}\right) }^{2m}.\]
\end{lemma}
\begin{rem}
It is also true that
\[ S(a,\widetilde \sigma_l^{(d,m)},d,m) \leq
\left\|M_{l-1}\right\|_{S_{2m}\left (H \otimes \ell^2(\N)^{\otimes {d-l}};
    H \otimes \ell^2(\N)^{\otimes l}\right) }^{2m},\] but we will only use
the inequality stated in the lemma. This inequality follows from the one
stated by conjugating by the rotation $k \in [2dm] \mapsto k+d$.
\end{rem}
\begin{proof}
  We fix $l \in \{0,\dots, d\}$. For any $s = (s_1,\dots, s_l) \in \N^l$ we
  denote by $A_s=(a_{s,t})_{t \in \N^{d-l}}$ viewed as a row matrix. As an
  operator, $A_s$ thus acts from $H \otimes \ell^2(\N)^{\otimes {d-l}}$ to
  $H$. For $s,s'\in \N^l$, if $r(k_1,\dots, k_d) = (k_d ,\dots, k_1)$
\[A_s A_{s'}^* = \sum_{t \in \N^{d-l}} a_{s,t} a_{s',t}^* = \sum_{t \in
  \N^{d-l}} a_{s,t} \widetilde a_{r(s',t)}^*.\]
Hence for $s^{(1)},s^{(2)},\dots, s^{(m)} \in \N^l$, if $s^{(m+1)}=s^{(1)}$,
\[ \prod_{i=1}^m A_{s^{(i)}} A_{s^{(i+1)}}^* = \sum_{t^{(1)},\dots, t^{(m)}
  \in \N^{d-l}} a_{s^{(1)},t^{(1)}} \widetilde a_{r(s^{(2)},t^{(1)})}^*
a_{s^{(2)},t^{(2)}} \widetilde a_{r(s^{(3)},t^{(2)})}^* \dots
 \widetilde a_{r(s^{(1)},t^{(m)})}^*.\]

 But for $k \in \N^{[2dm]}$, $k \respecte \sigma_l^{(d,m)}$ if and only if
 there exist $s^{(1)},s^{(2)},\dots, s^{(m)} \in \N^l$ and
 $t^{(1)},t^{(2)},\dots, t^{(m)} \in \N^{d-l}$ such that for all $i$,
 $(k_{2di+1}, k_{2di+2},\dots, k_{2di+d}) = (s^{(i)},t^{(i)})$ and $(k_{2di
   +2d}, k_{2di+2d-1},\dots, k_{2di+d+1}) = (s^{(i+1)},t^{(i)})$. Thus
 summing over $s^{(1)},s^{(2)},\dots, s^{(m)} \in \N^l$ in the preceding
 equation leads to
\[ \sum_{s^{(1)},s^{(2)},\dots, s^{(m)} \in \N^l} \prod_{i=1}^m A_{s^{(i)}}
A_{s^{(i+1)}}^* = \sum_{(k_1,\dots, k_{2m})
  \respecte \sigma_l^{(d,m)}} a_{k_1} \widetilde a_{k_2}^* a_{k_3} \dots
a_{k_{2m-1}} \widetilde a_{k_{2m}}^* .\] 

Taking the trace and using the trace property we get
\begin{eqnarray*} S(a,\sigma_l^{(d,m)},d,m) & =
  &\sum_{s^{(1)},s^{(2)},\dots, s^{(m)} \in \N^l} Tr\left(\prod_{i=1}^m
    A_{s^{(i)}}^* A_{s^{(i)}}
  \right)\\
  & = & Tr\left[ \left(\sum_{s \in \N^l} A_{s}^* A_{s}\right)^m
  \right]\\
  & = & Tr\left[ (M_l^*M_l)^m\right]
\end{eqnarray*}
where the last identity follows from the fact that $M_l = \sum A_s \otimes
e_{s1}$.  This concludes the proof for $\sigma_l^{(d,m)}$. For $\widetilde
\sigma_l^{(d,m)}$ with $1\leq l \leq d$, the same kind of computations
yield to
\[S(a,\widetilde \sigma_l^{(d,m)},d,m) = \sum_{s_l \in \N} Tr\left[
  \left(\sum_{s \in  \N^{l-1}} A_{(s,s_l)}^* A_{(s,s_l)}\right)^m
  \right].\]

To conclude we only have to use Lemma
\ref{lemma=inequality_by_interpolation} below.
\end{proof}

\begin{lemma}\label{lemma=inequality_by_interpolation}
Let $X_1,X_2 \dots X_N$ be matrices. Then for any integer $m \geq 1$
\[\sum_{i=1}^N Tr( (X_i^*X_i)^m) \leq Tr( (\sum_{i=1}^N X_i^*X_i)^m).\]
\end{lemma}
\begin{proof}
This is a general inequality for the non-commutative $L_p$-norms. Indeed,
for any $\alpha,N \in \N$, and $p\in [2,\infty]$,
the map
\[ \begin{array}{rccl}T:&M_{N,1}(M_\alpha(\C)) &\to&M_N(M_\alpha(\C)\\
  & \left(\begin{array}{c} X_1\\ \vdots \\X_N\end{array} \right) & \mapsto
  &\left(\begin{array}{ccc} X_1&0 & 0\\ 0& \ddots&0 \\0&0&X_N\end{array}
  \right) \end{array}\] is a contraction for all $p$-norms. For $p=2$, this
is easy because $T$ is an isometry. For $p = \infty$ this is also
obvious. For a general $p \in (2,\infty)$ the claim follows by
interpolation.

Applied for $p=2m$, this concludes the proof since for an integer $m$,
\[\left\|\left(\begin{array}{c} X_1\\ \vdots \\X_N\end{array}
  \right)\right\|_{2m}^{2m} = Tr( (\sum_{i=1}^N X_i^*X_i)^m)\]
and 
\[\left\|\left(\begin{array}{ccc} X_1&0 & 0\\ 0& \ddots&0 \\0&0&X_N\end{array}
  \right)\right\|_{2m}^{2m} = \sum_{i=1}^N Tr( (X_i^*X_i)^m).\]
\end{proof}

We are now able to state and prove the main result of this section. Recall
that for a partition $\pi$ of $NC^*(1,m)$, $B(\pi)$ was defined in part
\ref{part=invariants_of_P_k} as the number of blocks of the partition
$\Phi(\pi)$ (the map $\Phi$ was defined after Corollary
\ref{corollary=symmetrization_for_n}).
\begin{corollaire}
\label{corollary=operator_Cauchy-Schwarz_inequality}
  Let $\pi \in NC^*(d,m)$. Then if $a$ and $M_l$ are as in Lemma
  \ref{lemma=identifications},
  \[ |S(a,\pi,d,m)| \leq \prod_{l=0}^d \left\|M_l\right\|_{S_{2m}\left (H
      \otimes \ell^2(\N)^{\otimes {d-l}}; H \otimes \ell^2(\N)^{\otimes
        l}\right) }^{2m\mu_l}\] where $\mu_l=
  \big(B(\pi{\left|_{A_{l+1}}\right.}) - B(\pi{\left|_{A_{l}}\right.})\big)
  / (m-1)$ where we take the convention that $B(\pi{\left|_{A_0}\right.}) =
  1$ and $B(\pi{\left|_{A_ {d+1}}\right.}) = m$.
\end{corollaire}
\begin{proof}
  The idea is, as in Lemma 2 and Corollary 3 of \cite{MR1812816}, to
  iterate the inequality of Lemma \ref{lemma=CauchySchwarz}, except that
  here the combinatorial invariants of the map $\pi \mapsto (P_{kd}(\pi) ,
  P_{kd + md}(\pi))$ (Lemma \ref{lemma=combinatorial_invariant}) allow us
  to precisely determine the exponents of each $\|M_l\|_{2m}$. In the rest
  of the proof since no confusion is possible, we will simply denote
  $\sigma_l = \sigma_l^{(d,m)}$ and $\widetilde \sigma_l = \widetilde
  \sigma_l^{(d,m)}$, and $S$ will denote the set $\{\sigma_l, 0 \leq l \leq
  d\} \cup \{\widetilde \sigma_l, 0 \leq l \leq d\}$. Fix $\pi \in NC^*(d,m)$.

  Maybe the clearest way to write out a proof is using the basic
  vocabulary of probability theory (for a reference see for example
  \cite{MR1199812}).  Let us consider the (homogeneous) Markov chain
  $(\pi_n)_{n \geq 0}$ on (the finite state space) $NC^*(d,m)$ given
  by $\pi_0=\pi$ and $\pi_{n+1}=P_{id}(\pi_n)$ where $i$ is uniformly
  distributed in $[2m]$ and independent from $(\pi_k)_{0\leq k\leq n}$
  (note that $\pi_{n+1} \in NC^*(d,m)$ if $\pi_n \in NC^*(d,m)$ by
  Lemma \ref{lemma=properties_of_symmetrization_forNCnm}). Corollary
  \ref{corollary=symmetrization_for_n} implies that the sequence
  $(\pi_n)_n$ is almost surely eventually equal to one of the
  $\sigma_l$ or $\widetilde \sigma_l$. Its second statement indeed
  expresses that if $\pi_n \in S$ then $\pi_N=\pi_n$ for all $N \geq
  n$; it suffices therefore to prove that $p_n \egdef \mathbb P(\pi_n
  \notin S) \to 0$ as $n \to \infty$. But if $k$ is fixed with
  $2^{k-2} \geq m$, its first statement implies that $p_k \leq
  1-(1/2m)^k=c<1$ for any starting state $\pi_0$. From the equality
  $p_{n+k} = p_n \mathbb P(\pi_{n+k} \notin S \big | \pi_n \notin S)$
  and the Markov property we get that $p_{n+k} \leq c p_n$ for any
  integer $n \in \N$, from which we deduce that $p_n \leq c^{\lfloor
    n/k\rfloor} \to 0$ as $n \to \infty$.

  Let us denote $\lambda_l(\pi) = \mathbb P\left(\lim_n \pi_n =
    \sigma_l\right)$ and $\widetilde \lambda_l(\pi) = \mathbb P\left(\lim_n
    \pi_n = \widetilde \sigma_l\right)$ for $0\leq l \leq d$ (take
  $\widetilde \lambda_0(\pi)=0$); note that $\sum_l
  \lambda_l(\pi)+\widetilde \lambda_l(\pi)=1$.

  Lemma \ref{lemma=combinatorial_invariant} and the last statement of Lemma
  \ref{lemma=properties_of_symmetrization_forNCnm} show that for any $i \in
  \{1,\dots,d\}$ the sequence $B(\pi_n{\left|_{A_i}\right.})$ is a
  martingale. In particular since $\pi_0=\pi$, $B(\pi{\left|_{A_i}\right.})
  = \mathbb E \big[B(\pi_n{\left|_{A_i}\right.})\big]$ for any $n\geq
  0$. Letting $n \to \infty$ we get 
\begin{eqnarray*} B(\pi{\left|_{A_i}\right.}) & = & \sum_{l=0}^d
  \lambda_l(\pi) B(\sigma_l{\left|_{A_i}\right.})  + \sum_{l=1}^d
  \widetilde \lambda_l(\pi) B(\widetilde \sigma_l{\left|_{A_i}\right.})\\ &
  =& \sum_{l=0}^d \left(\lambda_l(\pi)+\widetilde
    \lambda_l(\pi)\right)\left( 1+ (m-1)1_{l<i}\right) \\ &=& 1+ (m-1)
  \sum_{0\leq l<i} \lambda_l(\pi)+\widetilde
    \lambda_l(\pi).
\end{eqnarray*}
We used the fact that $B(\sigma_l{\left|_{A_i}\right.}) = B(\widetilde
\sigma_l{\left|_{A_i}\right.}) = 1 + (m-1)1_{l<i}$. This follows from the
observations that since $\Phi(c_m)=\Phi(1_{2m})=1_m$, $B(c_m)= |1_m|=1$ and
that since $\Phi(r_m)=0_m$, $B(r_m)=m$. Subtracting the equalities above
for $i$ and $i+1$ gives
\begin{equation}
\label{eq=identification_des_lambda_l}
 (\lambda_i (\pi)+ \widetilde \lambda_i(\pi)) (m-1) =
B(\pi{\left|_{A_{i+1}}\right.}) - B(\pi{\left|_{A_{i}}\right.}) 
\end{equation} 
with the convention that $B(\pi{\left|_{A_{0}}\right.}) = 1$ and
$B(\pi{\left|_{A_{d+1}}\right.}) = m$.

On the other hand Lemma \ref{lemma=CauchySchwarz} implies that the sequence
$M_n = \log |S(a,\pi_n,d,m)|$ is a submartingale. As above letting $n \to
\infty$ in the inequality $M_0 \leq \mathbb E[M_n]$ yields
\[ \log |S(a,\pi,d,m)| \leq \sum_{l=0}^d \lambda_l(\pi) \log
|S(a,\sigma_l,d,m)| + \sum_{l=1}^d \widetilde \lambda_l(\pi) \log
|S(a,\widetilde \sigma_l,d,m)|.\] 
If we denote simply by $\|M_l\|_{2m}$ the
quantity $\|M_l\|_{S_{2m}\left (H \otimes \ell^2(\N)^{\otimes {d-l}}; H
    \otimes \ell^2(\N)^{\otimes l}\right) }$, then by Lemma
\ref{lemma=identifications} this inequality becomes
\[\left|S(a,\pi,d,m)\right| \leq \prod_{l=0}^d \|M_l\|_{2m}^{2m
  (\lambda_l(\pi)+\widetilde \lambda_l(\pi))}.\]
This inequality, combined with \eqref{eq=identification_des_lambda_l},
concludes the proof.
\end{proof}

\section{Main result}
We are now able to prove the main results of this paper. We first treat the
``holomorphic'' setting (Theorems \ref{thm=main_result_for_the_free_group}
and \ref{thm=Main_theorem}) for which the results we get are completely
satisfactory.

\subsection{Holomorphic setting}
\label{part=holomorphic_setting_proof}
It is a generalization to operator coefficients of the main result of
\cite{MR2353703}. When the coefficients $a_k$ are taken to be scalars, the
techniques of our Theorem \ref{thm=Main_theorem} give a new proof and an
improvement of the theorem 1.3 of \cite{MR2353703}. In \cite{MR2353703},
Kemp and Speicher introduce free Poisson variables to get an upper bound,
whereas our proof is more combinatorial and lies is the study of
$NC^*(d,m)$ that is done is part \ref{susbection=study_of_NCstar_d_m}. We
refer to \cite{MR2266879} or to the paper \cite{MR2353703} for definitions
and facts on free cumulants and $\mathscr R$-diagonal operators. We just
recall that the $*$-distribution of a variable $c$ in a $C^*$-probability
space is characterized by its free cumulants, which are the family of
complex numbers $\cum_n[c^{\varepsilon_1},\dots,c^{\varepsilon_n}]$, for $n
\in \N$ and $\varepsilon_i \in \{1,*\}$. Moreover the $\mathcal R$-diagonal
operators are exactly the operators $c$ for which the cumulants
$\cum_n[c^{\varepsilon_1},\dots,c^{\varepsilon_n}]$ vanish except if $n$ is
even and if $1$'s and $*$'s alternate in the sequence
$\varepsilon_1,\dots,\varepsilon_n$. Since the family
$\lambda(g_1),\dots,\lambda(g_r)$ (where $g_1,\dots,g_r$ are the generators
of the free group $F_r$) form an example of $*$-free $\mathcal R$-diagonal
operators, Theorem \ref{thm=main_result_for_the_free_group} is a particular
case of Theorem \ref{thm=Main_theorem}, that is why do not include a proof.

\begin{proof}[Proof of Theorem \ref{thm=Main_theorem}]
The start of the proof is the same as in the proof of Theorem 1.3 of
\cite{MR2353703}, and was sketched in the Introduction. Fix $p=2m \in 2\N$.

As in \eqref{eq=def_of_S_apidm}, if $k = (k_1,\dots, k_d)\in \N^d$ denote by
$\widetilde a_k = a_{(k_d,\dots, k_1)}$ and $\widetilde c_k = c_{(k_d,\dots,
  k_1)} = c_{k_d} \dots c_{k_1}$. First develop the norms:
\begin{eqnarray*}
  \left\| \sum_{k \in \N^d} a_k \otimes c_k \right\|_{2m}^{2m}
  & = & \sum_{k_1,\dots, k_{2m} \in \N^d} Tr(a_{k_1} a_{k_2}^* \dots
  a_{k_{2m}}^*) \tau(c_{k_1}   c_{k_2}^* \dots c_{k_{2m}}^*)\\
  & = & \sum_{k_1,\dots, k_{2m} \in \N^d} Tr(a_{k_1} \widetilde a_{k_2}^* \dots
  \widetilde a_{k_{2m}}^*) \tau(c_{k_1}
  \widetilde c_{k_2}^* \dots \widetilde c_{k_{2m}}^*).
\end{eqnarray*}
Take $k_1,\dots, k_{2m} \in \N^d$; if $k_l = (k_l(1), k_l(2),\dots, k_l(d))$
then 
\[c_{k_1} \widetilde c_{k_2}^* \dots \widetilde c_{k_{2m}}^* =c_{k_1(1)}
c_{k_1(2)} \dots c_{k_1(d)} c_{k_2(1)}^* \dots c_{k_2(d)}^* \dots
c_{k_{2m}(d)}^*\] 
and by the fundamental property of cumulants:
\[ \tau(c_{k_1} \widetilde c_{k_2}^* \dots \widetilde c_{k_{2m}}^*) =
\sum_{\pi \in NC(2dm)} \cum_\pi[ c_{k_1(1)}, \dots ,c_{k_1(d)},
c_{k_2(1)}^*, \dots ,c_{k_2(d)}^*, \dots ,c_{k_{2m}(d)}^*].\]

Denote $k = (k_1,\dots, k_{2m}) \in (\N^{d})^{2m} \simeq \N^{2dm}$. Since
freeness is characterized by the vanishing of mixed cumulants (Theorem
11.16 in \cite{MR2266879}), $\cum_\pi[ c_{k_1(1)}, \dots ,c_{k_{2m}(d)}^*]$
is non-zero only if $k \respecte \pi$, and in this case we claim that it is
equal to $\cum_\pi[c_{d,m}]$ where
\begin{equation}
\label{eq=def_of_c_dm}
c_{d,m} = \overbrace{\underbrace{c, \dots, c}_{d} ,\underbrace{c^*, \dots,
    c^*}_{d}, \dots ,\underbrace{c, \dots ,c}_{d}, \underbrace{c^*, \dots,
    c^*}_{d}}^{2m\textrm{ groups}}.
\end{equation}
Relabel indeed the sequence $k_1(1),\dots,k_{2m}(d)$ by
$k_1,\dots,k_{2dm}$, and denote also by
$\varepsilon_1,\dots,\varepsilon_{2dm}$ the corresponding sequence of $1$'s
and $*$'s, in such a way that $\cum_\pi[ c_{k_1(1)}, \dots
,c_{k_{2m}(d)}^*] = \cum_\pi[\big(c_{k_i}^{\varepsilon_i}\big)_{1\leq i\leq
  2dm}]$ and $\cum_\pi[c_{d,m}] =
\cum_\pi[\big(c^{\varepsilon_i}\big)_{1\leq i\leq 2dm}]$. By the definition
of $\cum_\pi$, we have
\[\cum_\pi[\big(c_{k_i}^{\varepsilon_i}\big)_{1\leq i\leq 2dm}] =
\prod_{V\in\pi} \cum_{|V|}[(c_{k_i}^{\varepsilon_i})_{i \in V}]\] where the
products runs over by the blocks of $\pi$. Similarly \[\cum_\pi[c_{d,m}] =
\prod_{V \in \pi} \cum_{|V|}[(c^{\varepsilon_i})_{i \in V}].\] Our claim
thus follows from the observation that if $k \respecte \pi$ then for any
block $V$ of $\pi$ there is an index $s$ such that $k_i=s$ for all $i\in
V$, and the equality $\cum_{|V|}[(c_{s}^{\varepsilon_i})_{i \in V}] =
\cum_{|V|}[(c^{\varepsilon_i})_{i \in V}]$ expresses just the fact that $c$
and $c_s$ have the same $*$-distribution and therefore the same cumulants.

The next claim is that since $c$ is $\mathscr R$-diagonal,
$\cum_\pi[c_{d,m}]$ is non-zero only if $ \pi \in NC^*(d,m)$. Since with
the previous notation $\cum_\pi[c_{d,m}] = \prod_{V \in \pi}
\cum_{|V|}[(c^{\varepsilon_i})_{i \in V}]$, this amounts to showing that if
there is a block $V$ of $\pi$ which is not of even cardinality or for which
$1$'s and $*$'s do not alternate in the sequence $(\varepsilon_i)_{i \in
  V}$, then $\cum_{|V|}[(c^{\varepsilon_i})_{i \in V}]=0$. But this is
exactly the definition of $\mathcal R$-diagonal operators. Thus we get
\[
  \left\| \sum_{k \in \N^d} a_k \otimes c_k \right\|_{2m}^{2m}  = 
  \sum_{\pi \in NC^*(d,m)} \cum_\pi[c_{d,m}] \sum_{(k_1,\dots, k_{2m}) \respecte \pi}
  Tr(a_{k_1} \widetilde a_{k_2}^* \dots \widetilde a_{k_{2m}}^*),
\]
or with the notation introduced in \eqref{eq=def_of_S_apidm}
\begin{equation}
\label{eq=formula_of_2m_norms_in_terms_of_cumulants_inside_the_proof} 
\left\| \sum_{k \in \N^d} a_k \otimes c_k \right\|_{2m}^{2m}  =  \sum_{\pi \in NC^*(d,m)} \cum_\pi[c_{d,m}] S(a,\pi,d,m).
\end{equation}

Up to this point we have mainly reproduced the beginning of the proof of
Theorem 1.3 of \cite{MR2353703} (the authors of \cite{MR2353703} only deal
with scalar $a_k$'s but there is no other difference). 

We can now use the study of $NC^*(d,m)$ that we did in part
\ref{susbection=study_of_NCstar_d_m}. Recall in particular that there is a
map $\mathcal P:NC^*(d,m) \rightarrow NC(m)^{(d)}$ the properties of which
are summarized in Theorem \ref{thm=decomposition_of_NC*nm}.

Take $(\sigma_1, \dots ,\sigma_d) \in NC(m)^{(d)}$ and denote $\mu_l =
(|\sigma_{l+1}| - |\sigma_{l}|)/(m-1)$ where $|\sigma|$ denotes the number
of blocks of $\sigma$ with the convention $|\sigma_0|=1$ and
$|\sigma_{d+1}|=m$. If $\pi \in NC^*(d,m)$ and $\mathcal P(\pi) =
(\sigma_1, \dots ,\sigma_d)$ then by Corollary
\ref{corollary=operator_Cauchy-Schwarz_inequality}, $|S(a,\pi,d,m)| \leq
\prod_{l=0}^d \left\|M_l\right\|_{2m}^{2m\mu_l}$.

Thus by the first part of Theorem \ref{thm=decomposition_of_NC*nm}, we have
that 
 \begin{multline*}\left|\sum_{\pi \in NC^*(d,m), \mathcal P(\pi) =
       (\sigma_1, \dots, \sigma_d)} 
\cum_\pi[c_{d,m}] S(a,\pi,d,m) \right|\\
   \leq 4^{2m} \prod_{l=0}^d \left\| M_l \right\|_{2m}^{2m\mu_l}
   \max_{\mathcal P(\pi) = (\sigma_1, \dots ,\sigma_d)} |\cum_\pi[c_{d,m}]| .
 \end{multline*}
 But by the second statement of Theorem \ref{thm=decomposition_of_NC*nm}
 and Lemma \ref{lemma=domination_of_cumulants_with_many_pairs} below
 (recall that for $\tau(c)=\cum_1[c]=0$ since $c$ is $\mathcal R$-diagonal)
\[|\cum_\pi[c_{d,m}]| \leq \|c\|_{2}^{2dm} \left(\frac{ 16
    \|c\|_{2m}}{\|c\|_{2}}\right) ^{4m},\]
which implies
 \begin{multline}
\label{eq=domination_of_sums_with_same_image_byP}
\left|\sum_{\pi \in NC^*(d,m), \mathcal P(\pi) =
       (\sigma_1, \dots ,\sigma_d)}
     \cum_\pi[c_{d,m}] S(a,\pi,d,m) \right|\\
   \leq 4^{10m} \prod_{l=0}^d \left\| M_l \right\|_{2m}^{2m\mu_l}
   \|c\|_{2}^{2dm} \left(\frac{ \|c\|_{2m}}{\|c\|_{2}}\right) ^{4m}.   
 \end{multline}

 But by Theorem 3.2 in \cite{MR583216}, for any non-negative integers
 $s_0,\dots, s_d$ such that $\sum_i s_i=m-1$, the number of
 $(\sigma_1,\dots,\sigma_d) \in NC(m)^{(d)}$ such that $|\sigma_{l+1}| -
 |\sigma_{l}|=s_l$ for any $0 \leq l \leq d$ (with the conventions
 $|\sigma_0|=1$ and $|\sigma_{d+1}|=m$) is equal to $(1/m) \binom m {s_0}
 \binom m {s_1} \dots \binom m {s_d}$. Thus from
 \eqref{eq=formula_of_2m_norms_in_terms_of_cumulants_inside_the_proof} we
 deduce
\begin{multline*}
 \left\| \sum_{k \in \N^d} a_k \otimes c_k \right\|_{2m}^{2m} \leq 4^{10m}
\|c\|_{2}^{2dm} \left(\frac{ \|c\|_{2m}}{\|c\|_{2}}\right) ^{4m} \\
\sum_{s_0 + \dots + s_d=m-1} (1/m) \binom m {s_0} \binom m {s_1} \dots \binom m
{s_d} \prod_{l=0}^d \left\| M_l \right\|_{2m}^{2m s_l/(m-1)} .
\end{multline*}

Denote for simplicity $\gamma_l= \left\| M_l \right\|_{2m}^{2m /(m-1)}$.
Since the number of $s_0, \dots ,s_d \in \N$ such that $s_0 + \dots +
s_d=m-1$ is equal to $\binom{m+d-1}{d}$, this inequality becomes
\begin{multline*} \left\| \sum_{k \in \N^d} a_k \otimes c_k
  \right\|_{2m}^{2m} \leq 4^{10m} \|c\|_{2}^{2dm} \left(\frac{
      \|c\|_{2m}}{\|c\|_{2}}\right) ^{4m} \binom{m+d-1}{d} 
  \\\sup_{s_0 + \dots +s_d=m-1} (1/m) \binom m {s_0} \binom m {s_1} \dots \binom m
{s_d} \prod_{l=0}^d \gamma_l^{s_l}.
\end{multline*}

Now use the fact that for any integers $N$ and $n$, $\binom N n \leq
(N/n)^n (N/(N-n))^{N-n}$ with the convention $(N/0)^0=1$. For a fixed $N$,
this can be proved by induction on $n \leq N/2$ using the fact that $x \in
\R^+ \mapsto x \log(1+1/x)$ is increasing. Thus
\[ \binom{m+d-1}{d} \leq \binom{m+d}{d} \leq \left(1+\frac{m}{d}\right)^{d}
\left(1+\frac{d}{m}\right)^{m}.\] 
But since $\log$ is concave, if $s_0+\dots +s_d=m-1$,
\begin{eqnarray*} \prod_{l=0}^d  \left(\frac{m}{m-s_l}\right)^{m-s_l} & = &
  \exp\left((md+1)\sum_0^d  \frac{m-s_l}{md+1}
    \log\big(m/(m-s_l)\big)\right)\\
& \leq &  \exp\left((md+1) \log\big(\sum_0^d  m/(md+1)\big)\right)\\
& = &  \exp\left((md+1) \log\big(1+(m-1)/(md+1)\big)\right) \leq \exp(m)
\end{eqnarray*}
and
\begin{eqnarray*}
  \prod_{l=0}^d \left(\frac{m\gamma_l}{s_l}\right)^{s_l} & = &
  \exp\left((m-1)\sum_0^d \frac{s_l}{m-1} \log\big( m\gamma_l
    /s_l\big)\right)\\
  & \leq & \exp\left((m-1) \log\big( m/(m-1) \sum_0^l \gamma_l \big)\right)\\
  & = & \left(\gamma_0+\dots \gamma_l\right)^{m-1} \left(\frac m
    {m-1}\right)^{m-1}
\end{eqnarray*}

But $(m/(m-1))^{m-1} \leq m$ for any $m \geq 1$. This leads to
\begin{multline} 
\label{eq=last_domination_of_2m-norm}
\left\| \sum_{k \in \N^d} a_k \otimes c_k
  \right\|_{2m}^{2m} \leq 4^{10m} \|c\|_{2}^{2dm} \left(\frac{
      \|c\|_{2m}}{\|c\|_{2}}\right) ^{4m} \\\left(1+\frac{m}{d}\right)^{d}
  \left(1+\frac{d}{m}\right)^{m} \exp(m) \left(\gamma_0+\dots
    \gamma_l\right)^{m-1}.
\end{multline}

Noting that since $2m/(m-1)\geq 2$,
\[(\gamma_0+\dots \gamma_l)^{m-1} = \left\| \left( \| M_l \|_{2m} \right)_l
\right\|_{\ell^{2m/(m-1)}(\{0,\dots ,d\})}^{2m} \leq \left\| \left( \| M_l
    \|_{2m} \right)_l \right\|_{\ell^{2}(\{0,\dots ,d\})}^{2m}\] and
taking the $2m$-th root in \eqref{eq=last_domination_of_2m-norm} one
finally gets
 \begin{multline*} 
   \left\| \sum_{k \in \N^d} a_k \otimes c_k \right\|_{2m} \leq 4^5 \sqrt{e
     (1+d/m)} \left(1+\frac{m}{d}\right)^{d/2m} \|c\|_{2}^d \left(\frac{
       \|c\|_{2m}}{\|c\|_{2}}\right) ^{2} \\ \left\| \left( \| M_l \|_{2m}
     \right)_l \right\|_{\ell^{2}(\{0,\dots ,d\})}.
 \end{multline*}
 To conclude for the case $m<\infty$, just note that
 $\left(1+\frac{m}{d}\right)^{d/m} \leq e$.

Letting $m \to \infty$ and noting that $\left(1+\frac{m}{d}\right)^{d/m}
\to 1$ concludes the proof for the operator norm.

When the $c_k$'s are circular, since $\cum_\pi[c_{d,m}]=1$ if $\pi \in
NC^*_2(d,m)$ and $\cum_\pi[c_{d,m}]=0$ otherwise, we can replace
\eqref{eq=domination_of_sums_with_same_image_byP} by 
\[\left|\sum_{\pi \in
NC^*(d,m), \mathcal P(\pi) = (\sigma_1, \dots ,\sigma_d)} \cum_\pi[c_{d,m}]
S(a,\pi,d,m) \right| \leq \prod_{l=0}^d \left\| M_l \right\|_{2m}^{2m\mu_l}.
\]
Following the rest of the arguments we get the claimed results.
\end{proof}
We still have to prove this Lemma that was used in the above proof.
\begin{lemma} 
\label{lemma=domination_of_cumulants_with_many_pairs}
Let $\pi \in NC(n)$ a non-crossing partition that has at least $K$ blocks
of size $2$ and in which all blocks have a size at most $N$.

  Let $c_1,\dots, c_n$ be elements of a tracial $C^*$-probability space
  $(\mathcal A,\tau)$ that are centered: $\tau(c_k)=0$ for all $k$.
  Let $m_p=\max_k \|c_k\|_p$ for $p=2,N$. Then

  \begin{equation}
\label{eq=majoration_of_cumulants_with_many_pairs}
 |\cum_\pi[c_1,\dots, c_n]| \leq m_2^{2K} \left(16 m_N \right)^{n-2K}.
\end{equation}
\end{lemma}
\begin{proof}
  Since both $\pi \mapsto \cum_\pi$ and the right-hand side of
  \eqref{eq=majoration_of_cumulants_with_many_pairs} are multiplicative, we
  only have to prove \eqref{eq=majoration_of_cumulants_with_many_pairs}
  when $\pi=1_n$ with $n \leq N$. Then as usual $\cum_\pi$ is denoted by
  $\cum_n$. If $n=1$ it is obvious since $\cum_1(c_1)=\phi(c_1) = 0$.

If $n=2$, then $K=1$ and $\cum_2(c_k,c_l) = \tau(c_k c_l) - \tau(c_k)
\tau(c_l) = \tau(c_k c_l)$. By the Cauchy-Schwarz inequality we get
$|\cum_2(c_k,c_l)| \leq m_2^2$.

We now focus on the case $n>2$, and then $K=0$. This is essentially done in
the proof of Lemma 4.3 in \cite{MR2353703} but we have to replace the
inequality $|\tau(c_{k_1} \dots c_{k_l})| \leq m_{\infty}^l$ by H{\"o}lder's
inequality $|\tau(c_{k_1} \dots c_{k_l})| \leq m_N^l$ for any $l \leq n
\leq N$. Following the proof of Lemma 4.3 in \cite{MR2353703}, we thus get
that
\[\cum_{n}[c_1,\dots, c_n] \leq 4^{n-1} \sum_{\sigma\in NC(n)}
{m_n}^n \leq 4^{2n} {m_N}^n.\]
\end{proof}

\subsection{Non-holomorphic setting}
Here we consider Theorems \ref{thm=non_holomo_Rdiag} and
\ref{thm=non_holomo_selfadj}. We only sketch their proofs. The idea is the
same as in the holomorphic setting, except that here the relevant subset of
non-crossing partitions is the set $NC(d,m)$ introduced and studied in part
\ref{part=NCdmsansstar}.

\begin{proof}[Sketch of proof of Theorem \ref{thm=non_holomo_selfadj}]
  We will use that if $c$ has a symmetric distribution, then $c$ has
  vanishing odd cumulants. This means that $\cum_\pi[c,\dots,c]=0$ unless
  $\pi$ has only blocks of even cardinality. To check this, by the
  multiplicativity of free cumulants, we have to prove that
  $\cum_n[c,\dots,c] = \cum_{1_n}[c,\dots,c]=0$ if $n$ is odd. But this is
  clear: since $-c$ and $c$ have the same distribution, $\cum_n[c,\dots,c]
  = \cum_n[-c,\dots,-c]$. On the other hand since $\cum_n$ is $n$-linear,
  $\cum_n[-c,\dots,-c] = (-1)^n\cum_n[c,\dots,c]$.

  Take $(c_k)_{k \in \N}$ and $(a_k)_{k \in \N^d}$ as in Theorem
  \ref{thm=non_holomo_selfadj} and define $\widetilde a_k$ and
  $c_{k_1,\dots,k_d}$ as in the proof of Theorem
  \ref{thm=Main_theorem}. Assume for simplicity that $c_k$ is normalized by
  $\|c_k\|_2 = 1$.  Denote by $I$ the set of $k=(k_1,\dots ,k_d) \in \N^d$
  such that for any $1\leq i<d$ $k_i \neq k_{i+1}$. Then for $p=2m$ we have
  that
    \[ \left\|\sum_{k \in I} a_k \otimes c_k
      \right\|_{2m}^{2m}  = \sum_{k_1,\dots, k_{2m} \in I} Tr(a_{k_1}
      \widetilde a_{k_2}^* \dots \widetilde a_{k_{2m}}^*) \tau(c_{k_1}
      c_{k_2} \dots c_{k_{2m}}).\]
Expanding the moment $\tau(c_{k_1} \dots c_{k_{2m}})$ using cumulants we get
\[ \tau(c_{k_1} c_{k_2}\dots c_{k_{2m}}) = \sum_{\pi \in NC(2dm)} \cum_\pi[
c_{k_1(1)}, \dots ,c_{k_1(d)}, c_{k_2(1)}, \dots ,c_{k_2(d)}, \dots
,c_{k_{2m}(d)}].\] By freeness of the family $(c_k)_{k\in\N}$, by the
assumption on the vanishing of odd moments and by Lemma
\ref{lemma=characterization_of_NCdm} such a cumulant is equal to $0$ except
if $\pi \in NC(d,m)$ and $(k_1,\dots,k_{2m}) \respecte \pi$, in which case
it is equal to $\cum_\pi[c,c\dots,c]$. We get
\[ \left\|\sum_{k \in I} a_k \otimes c_k \right\|_{2m}^{2m} = \sum_{\pi \in
  NC(d,m)} \cum_\pi[c,\dots,c] S(a,\pi,d,m).\] 

But by Lemma \ref{lemma=symmetrization_forNCdm}, Lemma
\ref{lemma=identifications} and an iteration of Lemma
\ref{lemma=CauchySchwarz} we get that for any $\pi \in NC(d,m)$
\[ S(a,\pi,d,m) \leq \max_{0\leq l \leq d} \|M_l\|_{2m}^{2m}.\]

On the other hand (remembering that $\|c\|_2=1$), Theorem
\ref{thm=decomposition_of_NCn_m} and Lemma
\ref{lemma=domination_of_cumulants_with_many_pairs} imply that for $\pi \in
NC(d,m)$,
\[|\cum_\pi[c,\dots ,c]| \leq \left(16\|c\|_{2m}\right) ^{4m}.\]

This yields
\[\left\|\sum_{k \in I} a_k \otimes c_k
\right\|_{2m}^{2m} \leq \sum_{\pi \in NC(d,m)} \left(16 \|c\|_{2m}\right)
^{4m} \max_{0\leq l \leq d} \|M_l\|_{2m}^{2m}.\]

But by Theorem \ref{thm=decomposition_of_NCn_m} $NC(d,m)$ has cardinality
less than $4^{2m}(d+1)^{2m}$. Taking the $2m$-th root in the preceding
equation we thus get 
\[\left\|\sum_{k \in I} a_k \otimes c_k
\right\|_{2m} \leq 4^5 (d+1) \|c\|_{2m}^2 \max_{0\leq l \leq d} \|M_l\|_{2m}.\]
This proves Theorem \ref{thm=non_holomo_selfadj} for the case when $p \in
2\N$. For $p = \infty$ just make $p \to \infty$.
\end{proof}

For Theorem \ref{thm=non_holomo_Rdiag} the proof is the same except that we
have to be slightly more careful in the beginning. Recall that $I_d$ is
the set of $(k_1,\varepsilon_1,\dots ,k_d,\varepsilon_d) \in (\N \times
\{1,*\})^d$ such that $\lambda(g_{k_1})^{\varepsilon_1} \dots
\lambda(g_{k_d})^{\varepsilon_d}$ corresponds to an element of length $d$
in the free group $F_\infty$. For a family of matrices
$(a_{k,\varepsilon})_{(k,\varepsilon)\in I_d}$ denote by
\[\breve a_{k,\varepsilon} = a_{(k_d,\dots,k_1),(\overline
  \varepsilon_d,\dots \overline \varepsilon_1)}\] where $\overline * = 1$
and $\overline 1=*$. The motivation for this notation is the following: for
$(k,\varepsilon) \in I_d$ denote by $c_{k,\varepsilon} =
c_{k_1}^{\varepsilon_1} \dots c_{k_d}^{\varepsilon_d}$, so that if $\breve
c_{k,\varepsilon}$ is defined as $\breve a_{k,\varepsilon}$, we have that
$\breve c_{k,\varepsilon}^* = c_{k,\varepsilon}$.

For $k=(k_1,\dots,k_{2m}) \in (\N^d)^{2m}$,
$\varepsilon=(\varepsilon_1,\dots,\varepsilon_{2m}) \in (\{1,*\})^{2m}$ and
$\pi \in NC(2dm)$ with blocks of even cardinality we will also write
$(k,\varepsilon) \respecte \pi$ if $k_i=k_j$ for all $i \sim_\pi j$ and if in addition for each block 
$\{i_1<\dots<i_{2p}\}$ of $\pi$, $1$'s and $*$'s alternate in the sequence $\varepsilon_{i_1},\varepsilon_{i_2},\dots,\varepsilon_{i_{2p}}$.

Last we denote, for $\pi \in NC(d,m)$ 
\[\widetilde S(a,\pi,d,m) = \sum_{(k,\varepsilon) \respecte \pi}
Tr(a_{k_1,\varepsilon_1} \breve a_{k_2,\varepsilon_2}^*
a_{k_3,\varepsilon_3} \dots \breve a_{k_{2m},\varepsilon_{2m}}^*).\]

The proofs of Lemma \ref{lemma=CauchySchwarz} and Lemma
\ref{lemma=identifications} still apply with this notation:
\begin{lemma}\label{lemma=CauchySchwarzbis}
  Let $\pi\in NC(d,m)$, and take a finitely supported family of matrices
  $a= (a_{k,\varepsilon})_{(k,\varepsilon) \in I_d}$ as above. For any integer
    $i$
    \[ \left|\widetilde S(a,\pi,d,m)\right| \leq \left(\widetilde
      S(a,P_{di}(\pi),d,m)\right)^{1/2} \left(\widetilde
      S(a,P_{(m+i)d}(\pi),d,m)\right)^{1/2}.\]
\end{lemma}
\begin{lemma}\label{lemma=identificationsbis}
  Let $d$, $m$, $a=(a_{k,\varepsilon})_{(k,\varepsilon) \in I_d }$ and
  $M_l$ be as in Theorem \ref{thm=non_holomo_Rdiag}, and $\sigma_l$ and
  $\widetilde \sigma_l$ as defined in Corollary
  \ref{corollary=symmetrization_for_n}. Then for $l\in\{0,1,\dots, d\}$:
\[ \widetilde S(a,\sigma_l^{(d,m)},d,m) = \left\|M_l\right\|_{S_{2m}\left
    (H \otimes \ell^2(\N)^{\otimes {d-l}}; H \otimes \ell^2(\N)^{\otimes l}\right)
}^{2m}.\] Moreover for $l \in \{1,\dots, d\}$
    \[ \widetilde S(a,\widetilde \sigma_l^{(d,m)},d,m) \leq \left\|M_l\right\|_{S_{2m}\left
        (H \otimes \ell^2(\N)^{\otimes {d-l}}; H \otimes \ell^2(\N)^{\otimes l}\right) }^{2m}.\]
\end{lemma}
We leave the proofs to the reader.

\begin{proof}[Sketch of the proof of Theorem \ref{thm=non_holomo_Rdiag}]
  Use the same notation as above.  Take $m \in \N$. Then as for
  the self-adjoint case we expand the $2m$-norm as follows:
  \begin{multline*} \left\|\sum_{(k,\varepsilon) \in I_d} a_{k,\varepsilon}
      \otimes c_{k,\varepsilon} \right\|_{2m}^{2m} =
    \\\sum_{
      (k_1,\varepsilon_1), \dots ,(k_{2m},\varepsilon_{2m}) \in I_d }
    Tr(a_{k_1,\varepsilon_1} \breve a_{k_2,\varepsilon_2}^* \dots \breve
    a_{k_{2m},\varepsilon_{2m}}^*) \tau(c_{k_1,\varepsilon_1}
    c_{k_2,\varepsilon_2} \dots c_{k_{2m},\varepsilon_{2m}}).
\end{multline*}
By the freeness, the definition of $I_d$, Lemma
\ref{lemma=characterization_of_NCdm} and the fact that the $c_k$'s are
$\mathscr R$-diagonal, the expression of the moment
$\tau(c_{k_1,\varepsilon_1} \dots c_{k_{2m},\varepsilon_{2m}})$ becomes
simply
\[ \tau(c_{k_1,\varepsilon_1} \dots c_{k_{2m},\varepsilon_{2m}}) =
\sum_{\pi \in NC(d,m)} 1_{(k,\varepsilon)\respecte \pi}
\cum_\pi[c_{k_1(1)}^{\varepsilon_1(1)}, \dots
,c_{k_{2m}(d)}^{\varepsilon_{2m}(d)}].\] Where if $(k,\varepsilon)\respecte
\pi$ and $\alpha_n(c) = \cum_{2n}[c,c^*,c,c^*,\dots,c,c^*] =
\cum_{2n}[c^*,c,c^*,c,\dots,c^*,c] $ we have that
\[ \cum_\pi[c_{k_1(1)}^{\varepsilon_1(1)} \dots
c_{k_{2m}(d)}^{\varepsilon_{2m}(d)}] = \prod_{V\textrm{ block of }\pi}
\alpha_{|V|/2}(c).\] In particular this quantity (which we will abusively denote by
$\cum_\pi(c)$) does not depend on $(k,\varepsilon)$.
We therefore get 
\[ \left\|\sum_{k \in I} a_k \otimes c_k \right\|_{2m}^{2m} = \sum_{\pi \in
  NC(d,m)} \cum_\pi[c] \widetilde S(a,\pi,d,m).\] 

From this point the proof of Theorem \ref{thm=non_holomo_selfadj} applies except that we use Lemma
\ref{lemma=identificationsbis} and an iteration of Lemma
\ref{lemma=CauchySchwarzbis} instead of Lemma
\ref{lemma=identifications} and an iteration of Lemma
\ref{lemma=CauchySchwarz}.
\end{proof}

\subsection{Lower bounds} Here we get some lower bounds on the norms we
investigated before. For example the following minoration is classical:

\begin{lemma} \label{lemma=minoration_for_circular} Let $(c_k)_{k\in \N}$
  be circular $*$-free elements with $\|c\|_1=1$. Then for any finitely
  supported family of matrices $(a_{k_1,\dots,k_d})_{k_1,\dots,k_d \in \N}$
  the following inequality holds:
\[\|\sum_{k_1,\dots,k_d \in \N} a_{k_1,\dots ,k_d} \otimes c_{k_1} \dots
c_{k_d}\| \geq \max_{0\leq l \leq d} \|M_l\|.\]
\end{lemma}
\begin{proof}We use the following (classical) realization of free circular
  elements on a Fock space. Let $H=H_1 \oplus_2 H_2$ be a Hilbert space
  with an orthonormal basis given by $(e_k)_{k \in \N} \cup (f_k)_{k \in
    \N}$ ($(e_k)$ is a basis of $H_1$ and $(f_k)$ of $H_2$). Let $\mathcal
  F(H) = \C \Omega \oplus \oplus_{n\geq 1} H^{\otimes n}$ be the full Fock
  space constructed on $H$ and for $k \in \N$ $s(k)$ (resp. $\widetilde
  s(k)$) the operator of creation by $e_k$ (resp. $f_k$). Define finally
  $c_k=s_k + \widetilde s_k^*$. It is well-known that $(c_k)_{k \in \N}$
  form of $*$-free family of circular variables for the state $\langle\cdot
  \Omega,\Omega\rangle$ which is tracial on the $C^*$-algebra generated by
  the $c_k$'s.

  Let $K$ be the Hilbert space on which the $a_k$'s act ($K=\C^\alpha$ if
  $a_k \in M_\alpha(\C))$. Then if $P_k$ denotes the orthogonal projection
  from $\mathcal F(H) \to H_2^{\otimes k}$, for $0 \leq l \leq d$ the
  operator $(\id \otimes P_l) \circ \sum_{k_1,\dots,k_d \in \N}
  a_{k_1,\dots ,k_d} \otimes c_{k_1} \dots c_{k_d}{\left|_{K\otimes
        H_1^{\otimes d-l}}\right.}$ corresponds to $M_l$ if it is viewed as
  an operator from $K\otimes H_1^{\otimes d-l} \simeq K\otimes
  \ell^2(\N)^{\otimes d-l}$ to $K \otimes H_2^{\otimes l}\simeq K\otimes
  \ell^2(\N)^{\otimes l}$ for the identification $H_1 \simeq \ell^2$ and
  $H_2\simeq\ell^2$ with the orthonormal bases $(e_k)$ and $(f_k)$.

This proves the Lemma.
\end{proof}

  We also prove the following Lemma which was stated in the introduction.
\begin{lemma}
\label{lemma=norm_of_matrices_for_p_prime}
Let $p$ be a prime number and define $a_{k_1,\dots,k_d} = \exp(2i \pi k_1
\dots k_d/p)$ for any $k_i \in \{1,\dots ,p\}$.

Then $\|(a_k)\|_2=p^{d/2}$ and for any $1 \leq l \leq d-1$ the matrix $M_l$
defined by $M_l = \left(a_{(k_1,\dots,k_l),(k_{l+1},\dots,k_d)}\right) \in
M_{p^l,p^{d-l}}(\C)$ satisfies $\|M_l\| \leq p^{d/2}\sqrt{(d-1)/ p}$.
\end{lemma}
\begin{proof}
  Since $\|M_l\|^2 = \|M_lM_l^*\|$ we compute the matrix $M_l M_l^* \in
  M_{p^{l},p^{l}}(\C)$.

For any $s=(s_1,\dots s_l)$ and $t=(t_1,\dots,t_l)\in \{1,\dots,p\}^l$ the
$s,t$-th entry of $M_l M_l^*$ is equal to 
\[ \sum_{(k_{l+1},\dots ,k_d) \in \{1,\dots,p\}^{d-l}} \exp\left(2i \pi
  (s_1\dots s_l- t_1\dots t_l) k_{l+1}\dots k_d/p\right).\] If $s_1\dots
s_l= t_1\dots t_l \mod p$ then this quantity is equal to $p^{d-l}$ whereas
otherwise, $\omega=\exp\left(2i \pi (s_1\dots s_l- t_1\dots t_l) /p\right)$
is a primitive $p$-th root of $1$, and it is straightforward to check that
for such an $\omega$,
\begin{eqnarray*} \sum_{(k_{l+1},\dots ,k_d) \in \{1,\dots,p\}^{d-l}}
  \omega^{ k_{l+1}\dots k_d} &=& \sum_{k_{l+1},\dots ,k_{d-1}}
  \sum_{k_d=1}^p \left(\omega^{
      k_{l+1}\dots k_{d-1}}\right)^{k_d}\\
  & = & \sum_{k_{l+1},\dots ,k_{d-1}} p 1_{k_{l+1}\dots k_{d-1}=0 \mod
    p}\\
  & = & p(p^{d-l-1}-(p-1)^{d-l-1}) .
\end{eqnarray*}
We therefore have that
\[ M_l M_l^* = (p^{d-l}-p(p-1)^{d-l-1}) \Big(1\Big)_{s,t \in [p]^l} +
p(p-1)^{d-l-1} \Big(1_{s_1\dots s_l=t_1\dots t_l}\Big)_{s,t \in [p]^l}.\]
The norm of an $N \times N$ matrix with entries all equal to $1$ is $N$.

Moreover if $[p]^l=\{(s_1,\dots ,s_l)\}$ is decomposed depending on the
value of $s_1\dots s_l$ modulo $p$, the matrix $\Big(1_{s_1\dots
  s_l=t_1\dots t_l}\Big)_{s,t \in [p]^l}$ is a block-diagonal matrix with
blocks having all entries equal to $1$. Its norm is therefore equal to
\begin{multline*}\max_{i\in [p]} \left|\left\{(s_1,\dots,s_l) \in [p]^l,
      s_1 \dots s_l=i \mod p\right\}\right| \\ =
      \left|\left\{(s_1,\dots,s_l) \in [p]^l, s_1 \dots s_l=0
      \right\}\right| = p^l - (p-1)^l.
\end{multline*}
By the triangle inequality the norm of $M_l M_l^*$ is thus less than
\begin{multline*} p^{l+1}(p^{d-l-1}-(p-1)^{d-l-1}) + p(p-1)^{d-l-1} (p^l -
  (p-1)^l) \\ = p^d -p (p-1)^{d-1} \leq (d-1) p^{d-1}
\end{multline*}
\end{proof}
\subsection*{Acknowledgment} I would like to thank Q. Xu for bringing
the problem to my attention and G. Pisier for many suggestions and
comments during discussions or after his careful reading of the many
preliminary versions of this paper. I would like also to thank the
referee for his useful suggestions regarding the exposition.

\bibliographystyle{alpha}
\bibliography{biblio}

\end{document}